\documentclass[12pt]{amsart}
\usepackage{amsmath, amscd, graphicx, latexsym, rlepsf, times, }

\usepackage{xparse}
\usepackage{float}
\usepackage{hyperref}
\usepackage{mathtools}
\DeclarePairedDelimiter{\ceil}{\lceil}{\rceil}
\newtheorem{theorem}{Theorem}[section]
\newtheorem*{theorem-non}{Theorem}
\newtheorem{prop}[theorem]{Proposition}
\newtheorem{lemma}[theorem]{Lemma}
\newtheorem{remark}[theorem]{Remark}

\newtheorem{definition}[theorem]{Definition}
\newtheorem{cor}[theorem]{Corollary}

\newcommand{\CPb}{\overline{\mathbb{CP}}^{2}}
\newcommand{\CP}{{\mathbb{CP}}^{2}}

\def \CPb {\overline{\mathbb{CP}}^{2}}
\def \CP {{\mathbb{CP}}^{2}}

\def \- {\setminus}

\title[On geography of simply connected symplectic $4$-manifolds]{On the geography of simply connected nonspin symplectic $4$-manifolds with nonnegative signature} 

\begin{document}

\author{Anar Akhmedov}
\address{School of Mathematics,
University of Minnesota,
Minneapolis, MN, 55455, USA}
\email{akhmedov@math.umn.edu}

\author{S\"{u}meyra Sakalli}
\address{School of Mathematics,
University of Minnesota,
Minneapolis, MN, 55455, USA}
\email{sakal008@math.umn.edu}

\begin{abstract} In \cite{AP3, AHP}, the first author and his collaborators constructed the irreducible symplectic $4$-manifolds 
that are homeomorphic but not diffeomorphic to $(2n-1)\CP\#(2n-1)\CPb$ for each integer $n \geq 25$, and the families of simply connected irreducible nonspin symplectic $4$-manifolds with positive signature that are interesting with respect to the symplectic geography problem. In this paper, we improve the main results in \cite{AP3, AHP}. In particular, we construct (i) an infinitely many irreducible symplectic and non-symplectic $4$-manifolds that are homeomorphic but not diffeomorphic to $(2n-1)\CP\#(2n-1)\CPb$ for each integer $n \geq 12$, and (ii) the families of simply connected irreducible nonspin symplectic $4$-manifolds that have the smallest Euler characteristics among the all known simply connected $4$-manifolds with positive signature and with more than one smooth structure. Our construction uses the complex surfaces of Hirzebruch and Bauer-Catanese on Bogomolov-Miyaoka-Yau line with $c_1^2 = 9\chi_h = 45$, along with the exotic symplectic $4$-manifolds constructed in \cite{A4, AP1, ABBKP, AP2, AS}.  

\end{abstract}


\maketitle

\section{Introduction} 

Let $X$ be a closed simply connected symplectic $4$-manifold, and $e(X)$ and $\sigma(X)$ denote the Euler characteristic and the signature of $X$, respectively. We define the following two invariants associated to $X$

\smallskip

\begin{center}

$\chi(X) := (e(X) + \sigma(X))/4$ and $c_{1}^{2} (X) := 2e(X) + 3\sigma(X)$

\end{center}

Recall that if $X$ is a complex surface, then $\chi(X)$ is equal to the holomorphic Euler characteristic $\chi_{h}(X)$ of $X$, while $c_1^2(X)$ is equal to the square of the first Chern class of $X$. A fundamental and challenging problem in the theory of complex surfaces (referred as the geography problem) is the characterization of all ordered pairs of integers $(a,b)$ that can be realized as ($\chi_{h}(X)$, $c_{1}^{2} (X)$) for some minimal complex surface $X$ of general type. The geography problem for complex surfaces was originally introduced and studied by Persson in \cite{Pe}, and further progress on this problem was made in \cite{MT, So, Ch, PPX, RU}. It seems presently out of reach to determine all such pairs $(a,b)$ that can be realized, even if one considers the simply complex surfaces with negative signature (see discussion in \cite{BHPV}, pages 291-93).

Since all simply connected complex surfaces are K\"ahler, thus symplectic, it is a natural problem to consider a similar problem for symplectic $4$-manifolds. The symplectic geography problem was originally introduced by McCarthy - Wolfson in \cite{MW}, refers to the problem of determining which ordered pairs of non-negative integers $(a, b)$ are realized as ($\chi(X)$, $c_1^2(X)$) for some minimal symplectic 4-manifold $X$. The geography problem of simply connected minimal symplectic $4$-manifolds has been first systematically studied in \cite{Go}, then studied subsequently in \cite{{FS5}, {PSz}, {JP}}). It was shown in \cite{{Go}, {FS5}, {JP}}) that many pairs ($\chi$, $c_1^2$) in negative signature region can be realized with non-spin symplectic $4$-manifolds, but there were finitely many lattice points with signature $\sigma < 0$ left unrealized. More recently, it was shown in \cite{ABBKP} and the subsequent work in \cite{AP2}, that all the lattice points with signature less than $0$ can be realized with simply connected minimal symplectic $4$-manifolds with odd intersection form. In terms of the symplectic geography problem, the work in \cite{ABBKP, AP2} concluded that there exists an irreducible symplectic $4$-manifold and infinitely many irreducible non-symplectic $4$-manifolds with odd intersection form that realize the following coordinates $(\chi,c_1^2)$ when $0 \leq c_1^2 < 8\chi$. A similar results for the nonnegative signature case were obtained in \cite{AP3, AHP}. We would like to remark that throughout this paper, we consider the geography problem for non-spin symplectic and smooth 4-manifolds. For the spin symplectic and smooth geography problems, we refer the reader to \cite{{PSz}, {AP5}} and references therein.


Our purpose in this article is to construct new non-spin irreducible symplectic and smooth $4$-manifolds with nonnegative signature that are interesting with respect to the symplectic and smooth geography problems. More specifically, we construct i) the infinitely many irreducible symplectic and infinitely many non-symplectic $4$-manifolds that all are homeomorphic but not diffemorphic to $(2n-1)\CP\#(2n-1)\CPb$ for any $n \geq 12$, and ii) the families of simply connected irreducible non-spin symplectic $4$-manifolds with positive signature that have the smallest Euler characteristics among the all known simply connected $4$-manifolds with positive signature and with more than one smooth structure. The building blocks for our construction are the complex surfaces of Hirzebruch and Bauer-Catanese on Bogomolov-Miyaoka-Yau line with $c_1^2 = 9\chi_h = 45$, obtained as $(\mathbb{Z}/5\mathbb{Z})^{2}$ covering of $\CP$ branched along a complete quadrangle \cite{BHPV, Main} (and their generalization in \cite{CD}), and the exotic symplectic $4$-manifolds constructed by the first author and his collaborators in \cite{A4, AP1, ABBKP, AP2, AS}, obtained via the combinations of symplectic connected sum and Luttinger surgery operations. We would like to point out that using our recipe and the family of examples in a very recent preprint of Catanese and Dettweiler \cite{CD}, one can generalize our construction to obtain examples of simply connected irreducible symplectic $4$-manifolds with positive signature that are interesting to the symplectic geography problem. This is explained in subsection~\ref{sig4}. 

Let $\CP$ denote the complex projective plane, with its standard orientation, and let $\CPb$ denote the underlying smooth $4$-manifold $\CP$ equipped with the opposite orientation. Our main results are stated as follows

\begin{theorem}\label{thm:main1}
Let $M$ be $(2n-1)\CP\#(2n-1)\CPb$ for any integer $n \geq 12$. Then there exist an infinite family of non-spin irreducible
symplectic\/ $4$-manifolds and an infinite family of irreducible non-symplectic\/ $4$-manifolds that all are homeomorphic but not diffeomorphic to\/ $M$.

\end{theorem}

The above theorem improves one of the main results in \cite{AP3} (see page 11) where exotic irreducible smooth structures on $(2n-1)\CP\#(2n-1)\CPb$ for $n\geq 25$ were constructed. Our next theorem improves the main results of \cite{AP3, AHP} for the positive signature case (see also the subsection~\ref{sig4}, where we delt with the cases of signature greater than $3$).

\begin{theorem}\label{thm:main2}
Let $M$ be one of the following\/ $4$-manifolds.  
\begin{itemize}

\item[(i)] $(2n-1)\CP\#(2n-2)\CPb$ for any integer $n \geq 14$.

\item[(ii)] $(2n-1)\CP\#(2n-3)\CPb$ for any integer $n \geq 13$.

\item[(iii)] $(2n-1)\CP\#(2n-4)\CPb$ for any integer $n \geq 15$.

\end{itemize}
Then there exist an infinite family of irreducible symplectic\/ $4$-manifolds and an infinite family of irreducible non-symplectic\/ $4$-manifolds that are homeomorphic but not diffeomorphic to\/ $M$.

\end{theorem}

The organization of our paper is as follows. In Section~\ref{HCB}, we introduce some background material on abelian covers and recall the construction of complex surfaces of Hirzebruch and Bauer-Catanese, with invariants $c_1^2 = 9\chi_h = 45$, that are obtained as an abelian covering of $\CP$ branched in a complete quadrangle. Furthermore, we prove a few results on these complex surfaces which will be needed later in the sequel. In Section~\ref{sur} we review the exotic non-spin symplectic and smooth $4$-manifolds with negative signature constructed by the first author and his collaborators in \cite{A4, ABBKP, AP2, AS}, which will serve as a second family of building block for our construction, and prove some lemmas about them that will be used in our proofs. The Sections~\ref{conI} and ~\ref{conII} are mostly devoted to the proofs of Theorem~\ref{thm:main1} and Theorem~\ref{thm:main2}, respectively. In Section~\ref{conII},
we also present the generalization of our examples to the cases of signature greater than $3$.

\section{Complex surfaces with ${c_{1}}^2 = 45$ and $\chi_h = 5$}{\label{HCB} 

In this section, we review the complex surfaces of Hirzebruch with invariants ${c_{1}}^2 = 45$ and $\chi_h = 5$ (see \cite{BHPV}, pages 240-42). These surfaces have been studied recently in the works of Bauer and Catanese (see \cite{Main}). These complex surfaces of general type are obtained as $(\mathbb{Z}/5\mathbb{Z})^2$ covers of $\mathbb{CP}^2$ branched in a complete quadrangle (cf. \cite{Main}) and sit on Bogomolov-Miyaoka-Yau line ${c_{1}}^2 = 9\chi_h$}. They will serve as one of the two building blocks in our construction of exotic simply connected non-spin $4$-manifolds with nonnegative signature, which will be obtained via the symplectic connected sum operation. Below, after recalling the construction of these complex surfaces (which employs the abelian covers) and computing their invariants, we will consider the fibration structure on them and derive some topological properties of these fibrations that will be used in our construction. 

\subsection{Abelian Covers}

In what follows, we recall basic definitions and properties of Abelian Galois ramified coverings. The proofs will be omitted, and the reader is referred to \cite{Par, BHPV} for the details.

\begin{definition}
Let $Y$ be a variety. An \emph{abelian Galois ramified cover} of $Y$ with abelian Galois group $G$ is a finite map $p: X \rightarrow Y$ with a faithful action of $G$ on $X$ such that $p$ exhibits $Y$ as the quotient of $X$ by $G$.
\end{definition}

We call such coverings \emph{abelian $G$-covers} and will assume that $Y$ is smooth and $X$ is normal. Let $R$ denote the ramification divisor of $p$ which consists of the points of $X$ that have nontrivial stabilizer. Indeed, $R$ is the critical set of $p$, and $p(R)$ is the branch divisor denoted by $D$. It is known that to every component of $D$, we can associate a cyclic subgroup $H$ of $G$ and a generator $\psi$ of $H^{*}$, the group of characters of $H$ (\cite{Par}, p195). We let $D_{H, \psi}$ be the sum of all components of $D$ which have the same group $H$ and character $\psi$. 

Now for an abelian $G$-cover  $p: X \rightarrow Y$ as above, and for any cyclic subgroup $H$ of $G$, let $g$ and $m_{H}$ denote the orders of $G$ and $H$, respectively. Then, the canonical classes of $X$ and $Y$ satisfy
\begin{equation}
{K_X}^2 = g \Big( K_Y + \sum\limits_{H, \psi}{\dfrac{m_H - 1} {m_H} D_{H, \psi}} \Big)^2
\label{eq:KK}
\end{equation}
where the sum is taken over the set $\mathcal{C}$ of cyclic subgroups of $G$ and for each $H$ in $\mathcal{C}$, the set of generators $\psi$ of $H^*$ (cf. \cite{Par}, Prop 4.2).   

Let us consider an abelian $G$-cover and let $D = \bigcup\limits_{i=1}^{k} D_i$ be its branch divisor with smooth irreducible components. Let $\chi: G \rightarrow \mathbb{Z}/d$ be a character of $G$ and $L_{\chi}$ be a divisor associated to the eigensheaf $\mathcal O(L_{\chi}) $. Then we have  (cf. \cite{Main})
\begin{equation}
d L_{\chi} =  \sum\limits_{i=1}^{k} \delta_i D_i,\,  \delta_i \in \mathbb{Z}/d\mathbb{Z} \simeq \{0,1,\dots,d-1\}.
\label{eq:unneces}
\end{equation}

\subsection{Construction of smooth surfaces with $K^2 = 45$ and $\chi_h = 5$}

Below we recall the construction of smooth algebraic surfaces with $K^2 = 45$ and $\chi_h = 5$, following \cite{Main}. These complex surfaces of general type are obtained as abelian covering of the complex plane branched over an arrangement of six lines shown as in Figure~\ref{fig:qu}, and were initially studied by Hirzebruch (cf. \cite{Hirze}, p.134).

\begin{figure}[ht]
\begin{center}
\includegraphics[scale=.50]{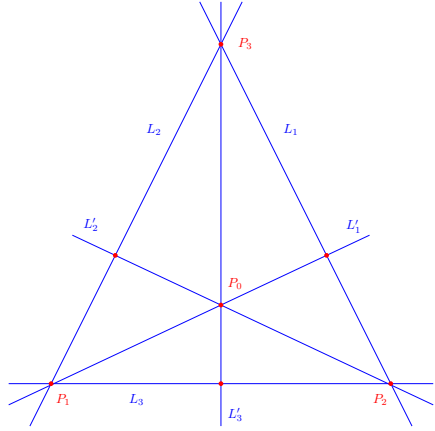}
\caption{Complete Quadrangle in $\mathbb{CP}^2$}
\label{fig:qu}
\end{center}
\end{figure}

In complex projective plane $\mathbb{CP}^2$ we take a complete quadrangle $\Delta$, which consist of the union of 6 lines through 4 points $P_0,\cdots,P_3$ in the general positions (see Figure~\ref{fig:qu}). Let us blow up $\mathbb{CP}^2$ at the points $P_0,\cdots,P_3$, and let $\pi: Y:= \widehat{\mathbb{CP}^2} \rightarrow \mathbb{CP}^2$ be the blow up map and $E_i$ be the exceptional divisor corresponding to the blow up at the point $P_i$ for $i=0,\cdots,3$. We introduce some notations now. In what follows, $i,j,k$ denote distinct elements of the set $\{1,2,3\}$. Let $H$ be the total transform in $Y$ of a line in $ \mathbb{CP}^2$, and let $\widetilde{L_{j}}$ and $\widetilde{L_{j}'}$ be the strict transform of the lines $L_j$ and $L_j'$ in $\mathbb{CP}^2$. That is to say,

\begin{equation}
\widetilde{L_j} = H- E_i-E_k,\, \widetilde{L_j'} = H-E_0-E_j.
\label{eq:lines}
\end{equation}
Let
\begin{equation}
D = \widetilde{L_1} + \widetilde{L_2} + \widetilde{L_3}+ \widetilde{L_1'}+ \widetilde{L_2'}+ \widetilde{L_3'} +E_0 + \cdots +E_3
\label{eq:D}
\end{equation}
be a divisor on $Y$ which has simple normal crossings and consists of the union of 10 lines (arising from the six lines of the quadrangle $\Delta$ and four exceptional divisors coming from the blow ups). Notice that $H,E_0,\cdots,E_3$ are generators of $H^2(Y, \mathbb{Z})$, and $H_1(Y-D,\mathbb{Z})$ is generated by $e_0,\cdots,e_3,l_1,l_2,l_3,l_1',l_2',l_3'$ with the relations
\begin{equation}\nonumber
e_0 = l_1'+l_2'+l_3',\, e_i= l_i+l_j'+l_k',\, \sum l_i + \sum l_i' = 0
\end{equation}
where $e_i$, $l_i$, $l_{i}'$ denote simple closed loops around $E_i$, $\tilde{L_i}$ and $\tilde{L_{i}'}$ respectively. Hence $H_1(Y-D,\mathbb{Z})$ is free group of rank 5. We know that a surjective homomorphism $\varphi: \mathbb{Z}^5 \simeq H_1(Y-D,\mathbb{Z}) \rightarrow (\mathbb{Z}/5\mathbb{Z})^2$ determines an abelian $(\mathbb{Z}/5\mathbb{Z})^2-$cover $p:S \rightarrow Y=\widehat{\mathbb{CP}^2}$. It can be shown that $p$ is branched exactly in $D$ given by $\eqref{eq:D}$. Since $D$ has simple normal crossings, the total space $S$ is smooth. 


Now for the total space $S$ of an abelian $(\mathbb{Z}/5\mathbb{Z})^2-$cover $p$ over $Y$, branched at $D$, we will show that $c_1^2(S)=K_S^2=45$ and $\chi_h(S) = 5$. Since the canonical class $K_Y$ of $Y$ is $-3H + \sum\limits_{i=0}^{3}E_i$, using the equation $\eqref{eq:KK}$, we compute
\begin{equation*}
K_S^2 = 5^2 \Big( (-3H + \sum\limits_{i=0}^{3}E_i) + \frac{4}{5} \sum\limits_{i=0}^{3}E_i +  \frac{4}{5} \sum\limits_{i=1}^{3}(L_i+L_i') \Big)^2
\end{equation*}
Next, using the relations in $\eqref{eq:lines}$, we get
\begin{eqnarray*}
K_S^2 &=& 5^2 \Big( (-3H + \sum\limits_{i=0}^{3}E_i) + \frac{4}{5} (6H-2E_0-2E_1-2E_2-2E_3) \Big)^2\\
&=& 5^2 \Big(\frac{9}{5} H - \frac{3}{5}  \sum\limits_{i=0}^{3}E_i \Big)^2
\end{eqnarray*}
Since $H \cdot E_i = 0,\, H^2=1$ and $E_i^2 = -1$, the above formula simplifies to: $K_S^2 = 9^2-4\cdot3^2 = 45$.

The Euler number $e(S)$ of $S$ can be found as follows.
 \begin{equation*}
e(S) = 25 e(\widehat{\mathbb{CP}^2} = \mathbb{CP}^2 \# 4\overline{\mathbb{CP}^2}) -20\cdot 10 e(\mathbb{CP}^1)+16 \cdot 15 = 15.
\end{equation*}

This equality follows from the inclusion-exclusion principle. In fact, if the degree 25 cover was unramified, we would have the Euler number $e = 25e(\widehat{\mathbb{CP}^2})$. Since for the lines in $D$ the cover is of degree 5, their contribution to $e(S)$ is $10 \cdot 5e(\mathbb{CP}^1)$. Therefore, we subtract $10 \cdot 20 e(\mathbb{CP}^1)$. But then for the points at the intersection of the lines in $D$, we need to add 16 times the Euler number of 15 points. Hence the above equality holds.

Finally, since $12 \chi_h(S) - c_1^2(S) = e(S)$, we have $\chi_h(S) = 5$.

\begin{remark} It is interesting to compare the above special construction with the more general constructions given in \cite{BHPV}, p.240 and \cite{Hirze}, p.134. In \cite{BHPV, Hirze}, using the  arrangements of $k$ lines in $\mathbb{CP}^2$ and taking their associated abelian $(\mathbb{Z}/n\mathbb{Z})^{k-1}$-covers, various algebraic surfaces were constructed. Notice that a partucular configuration with $k=6$ and $n=5$, leads to a surface $X(2)$ with $c_1^2(X(2)) = 5^5 (81/5^2-36/5^2)= 45 \cdot 5^3$ and $e(X(2))=15 \cdot 5^3$. Indeed, for the total spaces $X(m)$ of ($\mathbb{Z}/5\mathbb{Z})^m$-covers over the above configuration of 6 lines (where $m\geq2$), we have
\begin{equation}
c_1^2(X(m))=45 \cdot 5^{m-2} \quad \textrm{and} \quad e(X(m))=15 \cdot 5^{m-2}, for \ m \geq 2.  
\end{equation} 
\end{remark}

In \cite{Main}, Bauer and Catanese show that there are 4 nonisomorphic surfaces $S_1,S_2,S_3,S_4$ obtained from abelian $(\mathbb{Z}/5\mathbb{Z})^2-$covers over $Y$, branched at $D$, with invariants $K^2 = 45$ and $\chi_h = 5$. For $S_3$, we easily compute that $H^0(S,\mathcal{O}_S(K_S)) \simeq \mathbb{C} \oplus \mathbb{C} \oplus \mathbb{C} \oplus \mathbb{C}$ by using $ \eqref{eq:unneces}$ . Hence the geometric genus $p_g = \textrm{dim} H^0(S,\mathcal{O}_S(K_S))=4$. Furthermore, from the formula
\begin{equation*}
\chi_h = p_g - q +1
\end{equation*}
where $q$ is the regularity of the surface, we find that $q$ for $S_3$ is zero, hence $S_3$ is regular. Similarly, one can find that the irregularity of $S_i$, $i \in \{1, 2, 4\}$, is $2$. Therefore, only one of them is a regular surface. Let $S$ denote one of the surfaces  $S_i$, for $i \in \{1, 2, 4\}$.

\subsection{Fibration Structure on $S$}

In this subsection we analyze a well-known fibration structure on the complex surfaces $S$ constructed above with $q = 2$. Let $R_1$, $\cdots$, $R_{10}$ be the ramification divisors of $p:S \rightarrow Y$ lying over the lines $\widetilde{L_1'}$, $\widetilde{L_2'}$, $\widetilde{L_3'}$, $\widetilde{L_1}$, $\widetilde{L_2}$, $\widetilde{L_3}$, $E_0$, $\cdots$, $E_3$, respectively. Since $R_i^2=-1$ and $K_S \cdot R_i=3$, by the adjunction formula $K_S \cdot R_i + R_i^2 = 2g-2$, we see that the complex curves $R_i$'s have genus 2 for $i=1,\cdots,10$. Consider the map $p \circ \pi : S \rightarrow \mathbb{CP}^2$, where $\pi$ is the blow up map. Let $P$ be one of the four vertices of the complete quadrangle $\Delta$ in $\mathbb{CP}^2$ (see Figure~\ref{fig:qu}). The pencil of lines in $\mathbb{CP}^2$ passing through the point $P$ lifts to the fibration on $S$. Let us take one such point say, $P_3$ which is the intersection point of $L_2,L_3'$ and $L_1$ in $\Delta \subset \mathbb{CP}^2$. To determine the genus of the generic fiber of this fibration, we take a line $K$ passing through $P_3$ that is different than $L_2,L_3'$ and $L_1$ (see Figure~\ref{fig:fi2}). Observe that on $K$ there are 4 branch points. Furthermore, above each point on $K$ where no two lines intersect, there are 25/5 points (cf. \cite{BHPV}, p.241). Thus, the preimage of line $K-E_{3}$ in $Y$, which is the generic fiber of the given fibration, is a degree 5 cover of $K-E_{3}$ branched at 4 points. For the determination of the genus $g$ of the surface above $K-E_{3}$, we apply the Riemann-Hurwitz ramification formula

\begin{equation}
2g-2 = 5 (-2+4 \cdot \frac{4}{5}) \Rightarrow g=4.
\end{equation} 

Therefore, generic fibers are of genus 4 surfaces. Moreover, there are 4 distinct fibrations in $S$ coming from the points $P_i$'s, the vertices of the complete quadrangle.





Before proving the Proposition~\ref{prop1}, we state some well-known results in Complex Surface Theory that will be used in that proof. The proof of the first proposition can be found in [\cite{BHPV}, Proposition 11.4, page 118]. It is useful in determining the topological type of the singular fibers of the fibration on $S$ given above. 

\begin{figure}[ht]
\begin{center}
\includegraphics[scale=.60]{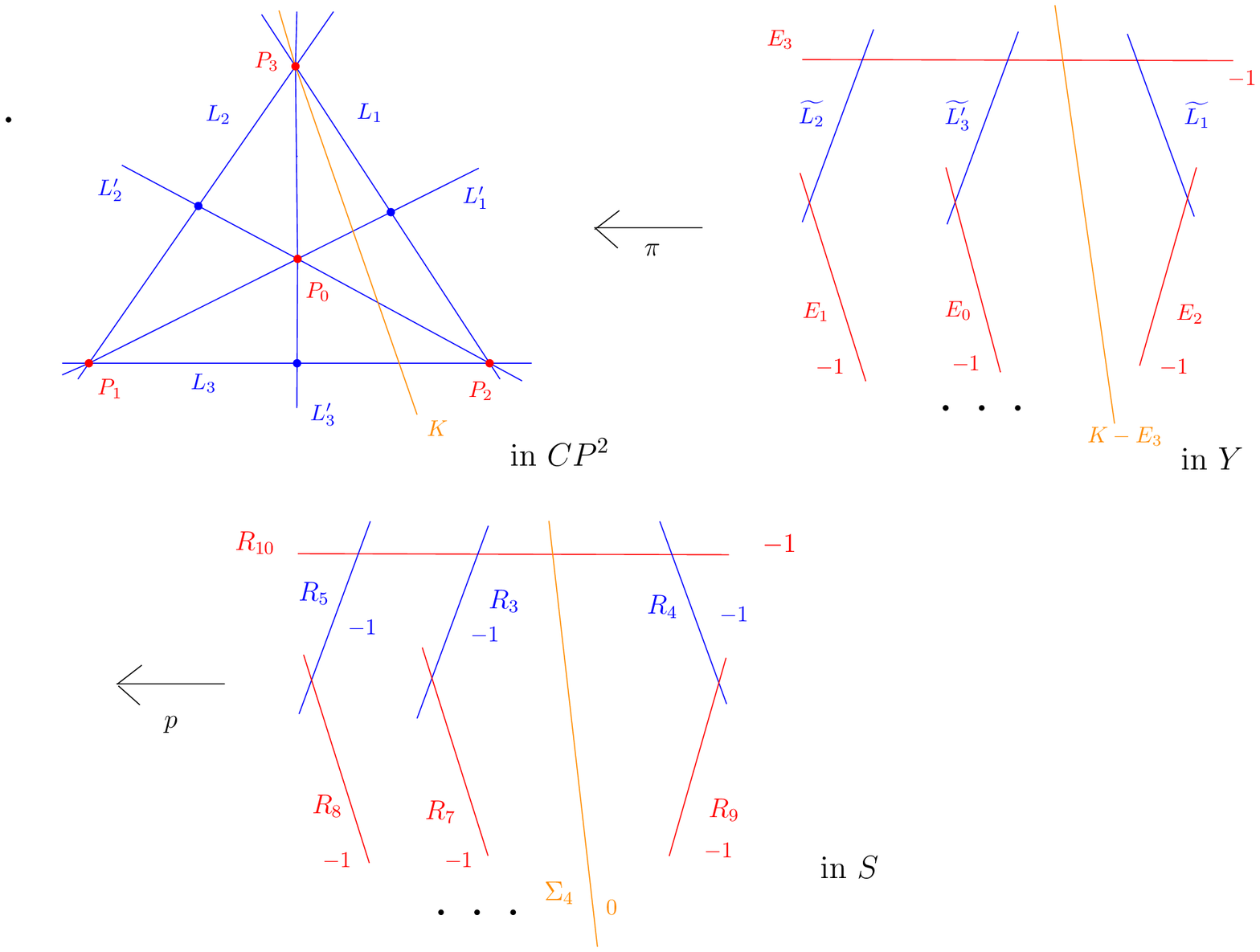}
\caption{Genus $4$ fibration on $S$ with 3 singular fibers}
\label{fig:fi2}
\end{center}
\end{figure}

\begin{prop}\label{fibers} Let $X$ be a compact connected smooth surface, and $C$ be a smooth connected curve. Let $f: X \rightarrow C$ be a fibration with $g > 0$, where $g$ is the genus of the general fiber $X_{s}$, and $X_{gen}$ a nonsingular fiber. Then 

\begin{itemize}

\item[(i)] $e(X_{s}) \geq e(X_{gen})$ for all fibers

\item[(ii)] $e(X_{s}) > e(X_{gen})$ for all singular fibers $X_{s}$, unless $X_{s}$ is a multiple fiber with $(X_{s})_{red}$ nonsingular elliptic curve

\item[(iii)] $e(X) = e(X_{gen}) \cdot e(C) + \sum (e(X_{s}) - e(X_{gen}))$

\end{itemize}

\end{prop}

Before stating the next proposition, we need to introduce some notations and facts. We follow the notations introduced in \cite{X, U}. Let $f: X \rightarrow C$ be a fibration, and $F = f^{-1}(c)$ denote a regular fiber of $f$. The inclusion map $i : F \xhookrightarrow{} X$ induces the homomorphism $i_*:\pi_{1}(F) \rightarrow \pi_1(X)$. Let us denote the image of $i_{*}$ by ${\mathcal{V}}_{f}$, and call it the \emph{vertical part} of $\pi_{1}(X)$. The following lemmas and corollary are not hard to prove (see \cite{U}, pages 13-14).

\begin{lemma}\label{vertical} ${\mathcal{V}}_{f}$ is a normal subgroup of $\pi_{1}(X)$, and is idependent of the choice of $F$.
\end{lemma}

Let us define the \emph{horizontal part} of $\pi_{1}(X)$ as ${\mathcal{H}}_{f}: = \pi_{1}(X)/{\mathcal{V}}_{f}$. Thus, we have $1 \rightarrow {\mathcal{V}}_{f} \rightarrow \pi_{1}(X) \rightarrow {\mathcal{H}}_{f} \rightarrow 1$.

Let us denote by $\{ x_{1}, \cdots, x_{s} \}$ the images of all the multiple fibers of $f$ (which maybe empty) and by $\{ m_{1}, \cdots, m_{s} \}$ their corresponding multiplicities. Let $C' = C \setminus  \{ x_{1}, \cdots, x_{s} \}$ and $\gamma_{i}$ be a small loop around the point $x_i$.

\begin{lemma}\label{vertical} The horizontal part ${\mathcal{H}}_{f}$ is the quotient of  $\pi_{1}(C')$ by the normal subgroup generated by the conjugates of ${\gamma_{i}}^{m_i}$ for all $i$.
\end{lemma}

The proposition given below was proved in \cite{U, X}.

\begin{prop}\label{mf} Let now F be any fiber of $f$ with multiplicity $m$. Then the image of $\pi_{1}(F)$ in $\pi_{1}(S)$ contains ${\mathcal{V}}_{f}$ as a normal subgroup, whose quotient group is cyclic of order $m$, which maps isomorphically onto the subgroup of ${\mathcal{H}}_{f}$ generated by the class of a small loop around the image of $F$ in $S$.
\end{prop}

The following corollary immediate consequence 

\begin{cor}\label{fibr} If $f$ has a section, then $1 \rightarrow {\mathcal{V}}_{f} \rightarrow \pi_{1}(X) \rightarrow C \rightarrow 1$.
\end{cor}

The proof of the next proposition can be found in \cite{Nor} (see Corollary 2.4 B, page). It essentially follows from Nori's work on Zariski's conjecture.

\begin{prop}\label{nori} Let $C$ be an embedded algebraic curve with $C^2>0$ in an algebraic surface $X$, then the induced group homomorphism $\pi_1(C)\rightarrow \pi_1(X)$ is surjective.
\end{prop}

Using the discussion above and Propositions \ref{fibers}, \ref{mf}, and \ref{nori}, we can prove the following 

\begin{prop}\label{prop1} Let $S$ be the surface (with $q=2$) given as above. Then the followings holds:

\begin{itemize}
\item[(i)] $S$ admits a genus $4$ fibration over genus $2$ surface with $3$ singular fibers
\item[(ii)] $S$ contains an embedded symplectic genus $6$ curve $R$ such that $\pi_1(R)\rightarrow \pi_1(S)$ is surjective.
\item[(iii)] $S\#\CPb$ contains an embedded symplectic genus $6$ curve $\widetilde{R}$ with self-intersection zero such that $\pi_1(\widetilde{R})\rightarrow \pi_1(S\#\CP)$ is surjective.
\end{itemize}
\end{prop}

\begin{proof} To prove (i) we consider the fibration given above, arising from the pencil of lines in $\mathbb{CP}^2$ passing through one of the vertices of the quadrangle $\Delta$. As we have shown above, the generic fiber of this fibration has genus $4$, and the ramification curves $\widetilde{L_1'}$, $\widetilde{L_2'}$, $\widetilde{L_3'}$, $\widetilde{L_1}$, $\widetilde{L_2}$, $\widetilde{L_3}$, $E_0$, $\cdots$, $E_3$ lifts to $-1$ complex curves $R_1$, $\cdots$, $R_{10}$ in $S$, respectively. Using the branched curves it is easy to see that the exceptional sphere $E_{3}$ lifts to a $-1$ curve $R_{10}$ in $S$. Thus, we have a fibration $f: S \rightarrow C$, where $C$ is a genus two curve. Furthemore, using the fact that $e(S)= 15$, $e(C) = -2$, $e(S_{gen}) = -6$ and Proposition~\ref{fibers}, we see that the fibration $f: S \rightarrow C$ has three singular fibers and each singular fiber has Euler characteristic $-5$. Furthemore, it is easy to see from the branched cover description of $S$ that each singular fiber has two irreducible components (arising from the curves $\widetilde{L_2} \cup E_{1}$ , $\widetilde{L_3'} \cup E_{0}$ , and $\widetilde{L_1} \cup E_{2}$ in $Y$), where each component is genus two curve of square $-1$ (see the Figure~\ref{fig:fi2})

(ii) The symplectic genus $6$ curve $R$ can be constructed in several ways: we simply can take one copy of a singular fiber $S_{sing}$, say $R_{3}  \cup R_{7}$ of the fibration $f: S \rightarrow C$, and $-1$ curve $R_{10}$, and resolve their transverse intersection point and also the single intersection point of the irreducible components $R_{3}$ and $R_{7}$ of $S_{sing}$, which are smooth genus two curves of square $-1$, of symplectically. Note that such a resolution can be done also holomorphically. The resulting curve $R$ has the self-intersection $R^{2} = (S_{sing} + R_{10})^{2} = 2S_{sing} \cdot R_{10} + {R_{10}}^{2} = 2 + (-1) = 1$. Using the lemmas above or 
Proposition \ref{nori}, we deduce that $\pi_1(R) \rightarrow \pi_1(S)$ is surjective. 

Alternatively, we can also construct such $R$ by resolving the transverse intersection points of the complex genus two curves $R_{10}$, $R_4$, and $R_9$ (see the Figure \ref{fig:fi2}). In this case, again we have $R^{2} = (R_{10} + R_{4} + R_{9})^{2} = 2(R_{10} \cdot R_{4} + R_{4} \cdot R_{5}) + {R_{10}}^{2} + {R_{5}}^{2} + {R_{4}}^{2} = 4 - 3 = 1$. Similarly as above, we can deduce that $\pi_1(R)\rightarrow \pi_1(S)$ is surjective.

We also refer a curious reader to the Section 5 in \cite{Main}, where the explicit computation of the fundamental group of $S$ given in Proposition 5.2, which relies on the work of Terada (see Theorem 5.1).

(iii) Let $\widetilde{R}$ be the symplectic genus six curve in $S\#\CPb$ obtained by blowing up $R$ at a point. Since $\widetilde{R}^{2} = 0$, the proof now simply follows from (ii).

\end{proof}

\section{Symplectic connected sum and Luttinger surgery}\label{sur}

The symplectic connected sum (cf. \cite{Go}) and Luttinger surgery (cf. \cite{Lu}, \cite{ADK}) operations have been very effective tools recently for constructing exotic smooth structures on $4$-manifolds \cite{A4, AP1, ABBKP, AP2, AS}. In what follows, we will briefly review the symplectic connected sum and Luttinger surgery operations, list some known results about them, and recall a few constructions of exotic $4$-manifolds with negative signatures obtained in \cite{A4, AP1, ABBKP, AP2, AS} via these operations, which we will use to build our exotic 4-manifolds with nonnegative signature later in the sequel. 

\subsection{Symplectic Connected Sum}{\label{sec: SCS}} Let us recall the definition and some basic facts about the symplectic connected operation. For the details, the reader is referred to \cite{Go}. 

\begin{definition} Let $(X_1, \ \omega_{1})$ and $(X_2, \ \omega_{2})$ be closed symplectic $4$-dimensional manifolds containing closed embedded surfaces $F_1$ and $F_2$ of genus $g$, with normal bundles $\nu_1$ and $\nu_2$, respectively. Assume that the Euler class of $\nu_i$ satisfy $e(\nu_1) +  e(\nu_2) = 0$. Then for any choice of an orientation reversing bundle isomorphism $\psi: \nu_1 \cong \nu_2$, the \emph{symplectic connected sum} of $X_1$ and $X_2$ along $F_1$ and $F_2$ is the smooth manifold $$X_1 \#_{\psi} X_2 = (X_1 - \nu_1)\cup_{\psi} (X_2 -\nu_2)$$.

\end{definition}

Note that the diffeomorphism type of $X_1 \#_{\psi} X_2$ depends on the choice of the embeddings and isomorphism $\psi$.  

\begin{theorem}\label{thm:symsum} The $4$-manifold $X_1 \#_{\psi} X_2$ admits a canonical symplectic structure $\omega$ induced by $\omega_{1}$ and $\omega_{2}$.

\end{theorem}

The Euler characteristic and the signature of the symplectic connected sum $X_1 \#_{\psi} X_2$ are easy to compute, and they are given by the following formulas:

\begin{equation}\label{eq: invariant I}
\begin{array}{l}
e(X_1 \#_{\psi} X_2)=e(X_1)+e(X_2)+4(g-1), \\ 
\sigma(X_1 \#_{\psi} X_2)=\sigma(X_1)+ \sigma (X_2)
\end{array}
\end{equation}

These formulas, in turn, imply the following formulas:

\begin{equation}\label{eq: invariants II}
\begin{array}{l}
\chi(X_1 \#_{\psi} X_2)=\chi(X_1)+\chi(X_2)+(g-1), \\ 
c_{1}^{2}(X_1 \#_{\psi} X_2) = c_{1}^{2}(X_1)+ c_{1}^{2} (X_2)+8(g-1)
\end{array}
\end{equation}

Next, we state a proposition which will be useful in the fundamental group computations of our examples obtained via the symplectic connected sum operation. The proof of this proposition can be found in \cite{Go} and \cite{Hal}.

\begin{prop}\label{prop:symsum} Let $X$ be closed, smooth $4$-manifold, and $\Sigma$ be closed submanifold of dimension $2$. Suppose that there exist a sphere $S$ in $X$ that intersects $\Sigma$ transversally in exactly one point, then the homomorphism $j_{*}: \pi_1(X \setminus \Sigma) \rightarrow \pi_1(X)$ induced by inclusion is an isomorphism. In particular, if $X$ is simply connected, then so is $X \setminus \Sigma$.
\end{prop}

\subsection{Luttinger surgery} \label{subsec:Luttinger} Let $(X, \omega)$\/ be a symplectic $4$-manifold, and $\Lambda$ be a Lagrangian torus embedded in $(X, \omega)$. It follows from the adjunction formula that the self-intersection number of $\Lambda$ is $0$, thus it has a trivial normal bundle. By Weinstein's Lagrangian neighborhood theorem, a tubular neighborhood $\nu \Lambda$ of $\Lambda$ in $X$ can be identified symplectically with a neighborhood of the zero-section in the cotangent bundle $T^*\Lambda \simeq T \times \mathbb{R}^2$ with its standard symplectic structure. Let $\gamma$ be any simple closed curve on $\Lambda$. The Lagrangian framing described above determines, up to homotopy, a push-off of $\gamma$ in $\partial(\nu \Lambda)$. Let $\gamma'$ is a simple loop on $\partial(\nu\Lambda)$ that is parallel to $\gamma$ under the Lagrangian framing.

\begin{definition} For any integer $m$, the $(\Lambda,\gamma,1/m)$ \emph{Luttinger surgery}\/ on $X$\/ is defined as $X_{\Lambda,\gamma}(1/m) = ( X \setminus \nu(\Lambda) ) \cup_{\phi} (\mathbb{S}^1 \times \mathbb{S}^1 \times \mathbb{D}^2)$, where, for a meridian $\mu_{\Lambda}$ of $\Lambda$, the gluing map $\phi : \mathbb{S}^1 \times \mathbb{S}^1 \times \partial \mathbb{D}^2 \to \partial(X \setminus \nu(\Lambda))$ satisfies $\phi([\partial \mathbb{D}^2]) = m[{\gamma'}] + [\mu_{\Lambda}]$ in $H_{1}(\partial(X \setminus \nu(\Lambda))$
\end{definition}

It is  shown in \cite{ADK} that $X_{\Lambda,\gamma}(1/m)$ possesses a symplectic form which agrees with the original symplectic form $\omega$ on $X\setminus\nu\Lambda$. The following lemma is easy to verify, the proof will be omitted.

\begin{lemma}\label{LSL} We have $\pi_1(X_{\Lambda,\gamma}(1/m)) = \pi_1(X- \nu \Lambda)/N(\mu_{\Lambda}
\gamma'^m)$, where $N(\mu_{\Lambda} \gamma'^m)$ denotes the normal
subgroup of $\pi_1(X- \nu \Lambda)$ generated by the product $\mu_{\Lambda}
\gamma'^m$. Moreover, we have $\sigma(X)=\sigma(X_{\Lambda,\gamma}(1/m))$,
and $e(X)=e(X_{\Lambda,\gamma}(1/m))$, where $\sigma$ and $\chi$ denote
the signature and the Euler characteristic, respectively.
\end{lemma}

\subsection{Luttinger surgeries on product manifolds $\Sigma_{n}\times \Sigma_{2}$ and $\Sigma_{n}\times \mathbb{T}^2$}\label{L}

In the following, we recall the construction of symplectic $4$-manifolds in \cite{AP2}, obtained from $\Sigma_{n}\times \Sigma_{2}$ and $\Sigma_{n}\times \mathbb{T}^2$ by performing a sequence of Luttinger surgeries along the Lagrangian tori. We use the same notations as in \cite{AP2} throughout this paper. The following two families of symplectic $4$-manifolds will be used as the building blocks in our construction. 

The first family of examples have the same cohomology ring as $(2n-3)(\mathbb{S}^2 \times \mathbb{S}^{2})$, and are constructed as follows. We fix integer $n\geq 2$, and denote by $Y_{n}$ the symplectic $4$-manifold obtained by performing $2n + 4$ Luttinger surgeries on $\Sigma_{n}\times \Sigma_{2}$, which consist of the following $8$ surgeries
\begin{eqnarray}\label{8 Luttinger surgeries}
&&(a_1' \times c_1', a_1', -1), \ \ \nonumber (b_1' \times c_1'', b_1', -1), \\ \nonumber
&&(a_2' \times c_2', a_2', -1), \ \ (b_2' \times c_2'', b_2', -1),\\ \nonumber
&&(a_2' \times c_1', c_1', +1), \ \ (a_2'' \times d_1', d_1', +1),\\ \nonumber
&&(a_1' \times c_2', c_2', +1), \ \ (a_1'' \times d_2', d_2', +1),
\end{eqnarray}
followed by the set of additional $2(n-2)$ Luttinger surgeries  
\begin{gather*}
(b_1'\times c_3', c_3',  -1), \ \ 
(b_2'\times d_3', d_3', -1), \\  
\dots,  \ \ \dots, \\
(b_1'\times c_n', c_n',  -1), \ \
(b_2'\times d_n', d_n', -1).
\end{gather*}

In the notation above, $a_i,b_i$ ($i=1,2$) and $c_j,d_j$ ($j=1,\dots,n$) denote the standard loops that generate $\pi_1(\Sigma_2)$ and $\pi_1(\Sigma_n)$, respectively. The Figure~\ref{fig:lagrangian-pair}, which was borrowed from \cite{AP2} (with a slight modification), depicts a typical Lagrangian tori along which the Luttinger surgeries are performed.  

\begin{figure}[ht]
\begin{center}
\includegraphics[scale=.89]{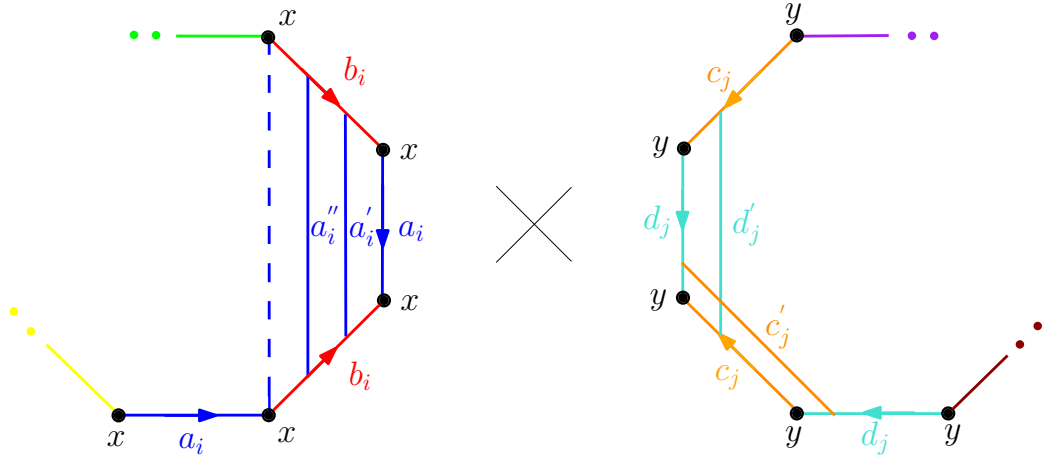}
\caption{Lagrangian tori $a_i'\times c_j'$ and $a_i''\times d_j'$}
\label{fig:lagrangian-pair}
\end{center}
\end{figure}

Using the Lemma~\ref{LSL}, we see that the Euler characteristic of $Y_{n}$ is $4n-4$ and the signature is $0$. Furthermore, the Lemma~\ref{LSL} implies that the fundamental group $\pi_1(Y_{n})$ is generated by loops $a_i,b_i,c_j,d_j$ ($i=1,2$ and $j=1,\dots,n$) and the following relations hold in $\pi_1(Y_{n})$:  
\begin{gather}\label{Luttinger relations}
[b_1^{-1},d_1^{-1}]=a_1,\ \  [a_1^{-1},d_1]=b_1,\ \  [b_2^{-1},d_2^{-1}]=a_2,\ \  [a_2^{-1},d_2]=b_2,\\ \nonumber
[d_1^{-1},b_2^{-1}]=c_1,\ \ [c_1^{-1},b_2]=d_1,\ \ [d^{-1}_2,b^{-1}_1]=c_2,\ \ [c_2^{-1},b_1]=d_2,\\ \nonumber
 [a_1,c_1]=1, \ \ [a_1,c_2]=1,\ \  [a_1,d_2]=1,\ \ [b_1,c_1]=1,\\ \nonumber
[a_2,c_1]=1, \ \ [a_2,c_2]=1,\ \  [a_2,d_1]=1,\ \ [b_2,c_2]=1,\\ \nonumber
[a_1,b_1][a_2,b_2]=1,\ \ \prod_{j=1}^n[c_j,d_j]=1,\\ \nonumber
[a_1^{-1},d_3^{-1}]=c_3, \ \ [a_2^{-1} ,c_3^{-1}] =d_3, \  \dots, \ 
[a_1^{-1},d_n^{-1}]=c_n, \ \ [a_2^{-1} ,c_n^{-1}] =d_n,\\ \nonumber
[b_1,c_3]=1,\ \  [b_2,d_3]=1,\ \dots, \
[b_1,c_n]=1,\ \ [b_2,d_n]=1.
\end{gather}

Note that the surfaces $\Sigma_2\times\{{\rm pt}\}$ and $\{{\rm pt}\}\times \Sigma_n$ in $\Sigma_2\times\Sigma_n$ are not affected by the above Luttinger surgeries, thus they descend to surfaces in $Y_{n}$. We will
denote these symplectic submanifolds by $\Sigma_2$ and $\Sigma_n$. Notice that we have $[\Sigma_2]^2=[\Sigma_n]^2=0$ and $[\Sigma_2]\cdot[\Sigma_n]=1$. Moreover, when $n \geq 3$, the symplectic 4-manifold $Y_{n}$ contains $2n-4$ pairs of geometrically dual Lagrangian tori.  These Lagrangian tori together with $\Sigma_2$ and $\Sigma_n$ generates the second homology group $H_{2}(Y_{n}) \cong \mathbb{Z}^{4n-6}$. 


Now we will consider a different family. Let us fix integers $n \geq 2$, $m \geq 1$, $p \geq 1$ and $q \geq 1$. Let $Y_{n}(1/p,m/q)$ denote smooth $4$-manifold obtained by performing the following $2n$ torus surgeries on $\Sigma_n\times \mathbb{T}^2$:

\begin{eqnarray}\label{eq: Luttinger surgeries for Y_1(m)} 
&&(a_1' \times c', a_1', -1), \ \ (b_1' \times c'', b_1', -1),\\  \nonumber
&&(a_2' \times c', a_2', -1), \ \ (b_2' \times c'', b_2', -1),\\  \nonumber
&& \cdots, \ \ \cdots \\ \nonumber
&&(a_{n-1}' \times c', a_{n-1}', -1), \ \ (b_{n-1}' \times c'', b_{n-1}', -1),\\  \nonumber
&&(a_{n}' \times c', c', +1/p), \ \ (a_{n}'' \times d', d', +m/q).
\end{eqnarray}

Let $a_i,b_i$ ($i=1,2, \cdots, n$) and $c,d$\/ denote the standard generators of $\pi_1(\Sigma_{n})$ and $\pi_1(\mathbb{T}^2)$, respectively. Since all the torus surgeries listed above are Luttinger surgeries when $m = 1$ and the Luttinger surgery preserves minimality, $Y_{n}(1/p,1/q)$ is a minimal symplectic $4$-manifold. The fundamental group of $Y_{n}(1/p,m/q)$ is generated by $a_i,b_i$ ($i=1,2,3 \cdots, n$) and $c,d$, and the Lemma~\ref{LSL} implies that the following relations hold in $\pi_1(Y_{n}(1/p,m/q))$:

\begin{gather}\label{Luttinger relations for Y_1(m)}
[b_1^{-1},d^{-1}]=a_1,\ \  [a_1^{-1},d]=b_1,\ \
[b_2^{-1},d^{-1}]=a_2,\ \  [a_2^{-1},d]=b_2,\\ \nonumber
\cdots,  \ \  \cdots,  \\ \nonumber
[b_{n-1}^{-1},d^{-1}]=a_{n-1},\ \  [a_{n-1}^{-1},d]=b_{n-1},\ \
[d^{-1},b_{n}^{-1}]=c^p,\ \ {[c^{-1},b_{n}]}^{-m}=d^q,\\ \nonumber
[a_1,c]=1,\ \  [b_1,c]=1,\ \ [a_2,c]=1,\ \  [b_2,c]=1,\\ \nonumber
[a_3,c]=1,\ \  [b_3,c]=1,\\ \nonumber
\cdots,  \ \  \cdots,  \\ \nonumber
[a_{n-1},c]=1,\ \  [b_{n-1},c]=1,\\ \nonumber
[a_{n},c]=1,\ \  [a_{n},d]=1,\\ \nonumber
[a_1,b_1][a_2,b_2] \cdots [a_n,b_n]=1,\ \ [c,d]=1.
\end{gather}

In this paper we will only consider the case $p = q = 1$. Let us denote by $\Sigma'_n, \Sigma'_{1} \subset Y_{n}(1,l)$ a genus $n$ surface and a torus that desend from the surfaces 
$\Sigma_{n}\times\{{\rm pt}\}$ and $\{{\rm pt}\}\times \mathbb{T}^2$ in $\Sigma_{n}\times \mathbb{T}^2$. The surfaces $\Sigma'_1$ and $\Sigma'_n$ generates the second homology group $H_{2}(Y_{n}(1,l)) \cong \mathbb{Z}^{2}$.

The following two theorems and the corollary derived from them are proved in \cite{AP3} and \cite{AHP} (see also \cite{ABBKP}, Theorem 23; \cite{AP2}, Theorem 2). We include them below for reader's convenience to make the exposition more self-contained.
 
\begin{theorem}\label{thm:gt}
Let $X$ be a closed symplectic 4-manifold that contains a symplectic torus $T$ of self-intesection $0$. Let $\nu T$ be a tubular neighborhood of $T$ and $\partial(\nu T)$ its boundary. Suppose that the homomorphism $\pi_{1}(\partial(\nu T)) \rightarrow \pi_{1}(X \setminus \nu T)$ induced by the inclusion is trivial. Then for any pair of integers $(\chi, c)$ satysfying 
\begin{equation}\label{eq: invariants}
\begin{array}{l}
\chi \geq 1 \; and \; 0 \leq c \leq 8\chi  
\end{array}
\end{equation}

\noindent there exist a symplectic 4-manifold $Y$ with $\pi_{1}(Y) = \pi_{1}(X)$, 
\begin{equation}\label{eq: invariants}
\begin{array}{l}
\chi_{h}(Y) = \chi_{h}(X) + \chi \; and \; {c_{1}}^{2}(Y) = {c_{1}}^{2}(X) + c 
\end{array}
\end{equation}

\noindent Moreover, if $X$ is minimal then $Y$ is minimal as well. If $c < 8\chi$, or $c = 8\chi$ and $X$ has an odd intersection form, then the corresponding $Y$ has an odd indefinite intersection form. 

\end{theorem}

The next theorem will be used to produce an infinite family of pairwise nondiffeomorphic, but homeomorphic simply connected $4$-manifolds.

\begin{theorem}\label{thm:gt1} Let $Y$ be a closed simply connected minimal symplectic $4$-manifold with $b_{2}^{+}(Y) > 1$. Assume that $Y$ contains a symplectic torus $T$ of self-intersection $0$ such that $\pi_{1}(Y\setminus T) = 1$. Then there exist an infinite family of pairwise nondiffemorphic irreducible symplectic $4$-manifolds and an infinite family of pairwise nondiffemorphic irreducible nonsymplectic 4-manifolds, all of which are homemorphic to $Y$.

\end{theorem}

The following corollary follows from the above Theorems, and proof can be found in \cite{AHP}.

\begin{cor}\label{cor1} 
Let $X$ be a closed simply connected nonspin minimal symplectic $4$-manifold with $b_{2}^{+}(X) > 1$ and $\sigma(X) \geq 0$. Assume that $X$ contains disjoint symplectic tori $T_{1}$ and $T_{2}$ of self-intersections $0$ such that $\pi_{1}(X\setminus(T_{1}\cup T_{2})) = 1$. Suppose that $\sigma$ is a fixed integer satisfying $0 \leq \sigma \leq \sigma(X)$. If $\ceil{x} = min\{k \in \mathbb{Z} | k \geq x \}$ and we define 

\begin{equation}\label{f}
\begin{array}{l}
l(\sigma) = \ceil[\bigg]{\frac{\sigma(X)-\sigma}{8} - 1} 
\end{array}
\end{equation}

\noindent then if $k$ is any odd integer satisfying $k \geq {b_{2}}^{+}(X) + 2l(\sigma) + 2$, then there exist an infinite family of pairwise nondiffemorphic irreducible symplectic $4$-manifolds and an infinite family of pairwise nondiffemorphic irreducible nonsymplectic 4-manifolds, all of which are homemorphic to $k\CP\#(k-\sigma)\CPb$
\end{cor}

\subsection{Symplectic Building Blocks}\label{sbb} In this section we collect some symplectic building blocks that will be used in our construction of exotic $4$-manifolds with nonnegative signature. The symplectic $4$-manifolds, with negative signature, given below were constructed by the first author and his collaborators in \cite{A4, AP2, AS}). 

Our first family of symplectic building blocks comes from \cite{AS} (see Theorem 5.1, page 14), though a few cases were treated in \cite{A4} (see Theorems 2, page 2). Let us state the Theorem 5.1 \cite{AS} in a special case that we will need.

\begin{theorem}\label{thm:main}
Let $M$\/ be \/$(2k-1)\CP\#(2k+3)\CPb$ for any $k \geq 1$. There exist a family of smooth closed simply-connected minimal symplectic\/ $4$-manifold and an infinite family of non-symplectic $4$-manifolds that is homeomorphic but not diffeomorphic to\/ $M$ that can obtained by a sequence of Luttinger surgeries and a single generalized torus surgery on Lefschetz fibrations. 
\end{theorem}

For the convenience of the reader, we sketch the construction given in \cite{AS} (in a special case $n = 1$), and direct the reader to this reference for full details. It is well known that the symplectic $4$-manifold $Y(k) = \Sigma_{k} \times \mathbb{S}^2\#4\CPb$ admits a genus $2k$ Lefschetz fibration over $\mathbb{S}^{2}$ with $2k+2$ vanishing cycles \cite{Ko}. One of the two building blocks of exotic $M$, given as in the statement of the theorem above, is $Y(k)$ with a genus $2k$ symplectic submanifold $\Sigma_{2k}\subset Y(k)$, a regular fiber of the Lefschetz fibration. We endowed $Y(k) = \Sigma_{k} \times \mathbb{S}^2\#4\CPb$ with the symplectic structure induced from the given Lefschetz fibration. The other building block of exotic $M$ is the smooth $4$-manifold $Y_{g}(1,m)$, along the submanifold $\Sigma'_g$ of genus $g$. Recall from (\cite{AS}, see pages 14-15), the manifold $Y_{g}(1,m)$ was obtained from the product $4$-manifold $\Sigma_{g} \times \mathbb{T}^{2}$ by performing appropriate $2g-1$ Luttinger, and one generalized torus surgeries, where we set $g = 2k$. Let $X(k,m)$ denote the smooth $4$-manifold obtained by forming the smooth fiber sum of $Y(k)$ and $Y_{g}(1,m)$\/ along the surfaces $\Sigma_{2k}$ and $\Sigma'_{g}$. We shall need the following Theorem proved in \cite{AS} (see proof of Theorem 5.1, pages 14-18), which summarize topological properties of the manifold $X(k,m)$.

\begin{theorem}\label{thm1}  
\begin{itemize}
\item[(i)] $X(k,m)$ is simply connected
\item[(ii)] $e(X(k,m))= 4k + 4$, $\sigma (X(k,m)) = - 4$, $c_1^{2}(X(n,k,m)) = 8k - 4$, and\/ $\chi(X(k,m)) = k$.
\item[(iii)] $X(k,m)$ is minimal symplectic for $m = \pm 1$ and non-symplectic for $|m| > 1$.
\item[(iv)] $X(k,m)$ contains the smooth surface $\Sigma$ of genus $2k$ with self-intersection $0$, and $4$ tori $T_i$ of self-intersection $-1$ intersecting $\Sigma$ positively and transversally. Moreover, if $m = \pm 1$, these submanifolds all are symplectic.
\end{itemize}
\end{theorem}

\begin{figure}[ht]
\begin{center}
\includegraphics[scale=.49]{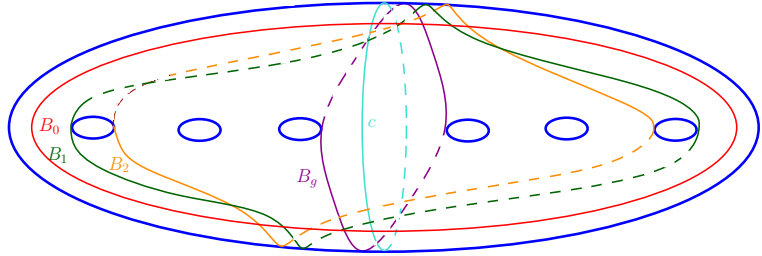}
\caption{Vanishing cycles of a genus $2k$ Lefschetz fibration on $Y(k)$}
\label{vanishing cycles}
\end{center}
\end{figure}

\begin{remark}\label{rem1}  

In addition to the surfaces given in (iv), $X(k,m)$ contains $2k-2$ disjoint rim tori $\bar{R_{i}}$ with self-intersections $0$, and their associated dual vanishing classes $V_{i}$ with self-intersection $-2$, and smooth surface $\Sigma_{g+1}$ with self-intersection $0$.
The rim tori $\bar{R_{i}}$ and their associated dual vanishing classes $V_{i}$ (which all are tori) arise from the generalized fiber sum of $Y(k)$ and $Y_{g}(1,m)$ along $\Sigma_{2k}$. Notice that the vanishing cycles $B_{3}, B_{4}, \cdots, B_{2k}$ bound the vanishing disk in $Y(k) \setminus \Sigma_{2k} \times \mathbb{D}^{2}$ and $-1$ torus in $Y_{g}(1,m) \setminus \Sigma_{2k} \times \mathbb{D}^{2}$. The second homology of $X(k,m)$ is generated by the classes of these $4k + 2$ surfaces. Furthemore, in $X(k,\pm 1)$ the surfaces $\bar{R_{i}}$ and $V_{i}$ are Lagrangian, and the rest of the surfaces are symplectic submanifolds.    

\end{remark}

We will use the case $(n,k) = (3,1)$, of the above result in our paper.

Our next symplectic building blocks comes from \cite{ABBKP} (see Theorem 5.1, page 14) 

\begin{theorem}\label{thm2} 
For any integer $g \geq 1$, there exist a minimal symplectic 4-manifold $X_{g,g+2}$ obtained via Luttinger such that
\begin{itemize}
\item[(i)] $X_{g,g+2}$ is simply connected
\item[(ii)] $e(X_{g,g+2})= 4g+2$, $\sigma (X_{g,g+2}) = - 2$, $c_1^{2}(X_{g,g+2}) = 8g-2$, and\/ $\chi(X_{g,g+2}) = g$.
\item[(iii)] $X_{g,g+2}$ contains the symplectic surface $\Sigma$ of genus $2$ with self-intersection $0$ and $2$ genus $g$ surfaces with self-intersection $-1$ intersecting $\Sigma$ positively and transversally.  
\end{itemize}
\end{theorem}

Our third symplectic building blocks comes from \cite{AP2}. 

\begin{theorem}\label{thm3}  
There exist a minimal symplectic 4-manifold $X_{g,g+1}$ obtained via Luttinger such that
\begin{itemize}
\item[(i)] $X_{g,g+1}$ is simply connected
\item[(ii)] $e(X_{g,g+1})= 4g+1$, $\sigma (X_{g,g+1}) = - 1$, $c_1^{2}(X_{g,g+2}) = 8g-1$, and\/ $\chi(X_{g,g+1}) = g$.
\item[(iii)] $X_{g,g+1}$ contains the symplectic surface $\Sigma$ of genus $2$ with self-intersection $0$, genus $\Sigma_{g+1}$ symplectic surface with self-intersection $0$ intersecting $\Sigma$ positively and transversally.  
\end{itemize}
\end{theorem}

For the convenience of the reader, we will spell out the details of the constructions of  $X_{g,g+2}$ and  $X_{g,g+1}$ in Section~\ref{conII}.

\section{Construction of exotic $(2n-1)\CP\#(2n-1)\CPb$ for $n \geq 12$}\label{conI}

In this section we intend to study the geography of non-spin simply connected symplectic and smooth $4$-manifolds with signature zero. We will prove our first main theorem (Theorem~\ref{thm:main1}), which improves the main result obtained in \cite{AP3}. We will split the proof of Theorem~\ref{thm:main} into two separate theorems. The first theorem (Theorem~\ref{thm:main.1}) deals with the case $n \geq 13$, and the second theorem (Theorem~\ref{thm:main.2}) addresses the case $n=12$, for which the construction is slightly different than $n \geq 13$ case.

The proof of Theorems~\ref{thm:main.1} and \ref{thm:main.2} will be broken into several parts. First, we construct our manifolds using the symplectic connected sum of the complex surface $S$, and the symplectic building blocks given in Section~\ref{sbb} obtained via Luttinger surgery. In the second step, we show that the fundamental groups of our manifolds are trivial, and determine their homeomorphism types. Next, using the Seiberg-Witten invariants and Usher's Minimality Theorem \cite{U}, we distinguish the diffeomorphism types of our $4$-manifolds from the standard $(2n-1)\CP\#(2n-1)\CPb$. Finally, by performing the knot surgery operation along a homologically essential torus on these symplectic $4$-manifolds, we obtain an infinite family of pairwise non-diffeomorphic irreducible symplectic and non-symplectic exotic copies of $(2n-1)\CP\#(2n-1)\CPb$.

\begin{theorem}\label{thm:main.1}
Let $M$\/ be \/$(2n-1)\CP\#(2n-1)\CPb$ for any $n \geq 13$. There exists an infinite family of smooth closed simply-connected minimal symplectic\/ $4$-manifolds and an infinite family of non-symplectic $4$-manifolds that all are homeomorphic but not diffeomorphic to\/ $M$. 
\end{theorem}

Our first building block will be the complex surface $S\#\CPb$ along with the genus $6$ symplectic surface $\widetilde{R} \subset S\#\CPb$, which we constructed in Section~\ref{HCB}. We endowed $S\#\CPb$ with the symplectic structure induced from the K\"{a}hler structure. Our second building block will be the symplectic $4$-manifold $X(3,1)$ along the symplectic submanifold $\Sigma'_6$ (see Section~\ref{sbb}). Let $Z(3)$ be the symplectic $4$-manifold obtained by forming the symplectic connected sum of $S\#\CPb$ and $X(3,1)$\/ along the surfaces $\widetilde{R}$ and $\Sigma'_{6}$. 

\begin{equation*}
Z(3) = (S \#  \overline{\mathbb{CP}}^{2}) \#_{\widetilde{R} = \Sigma_{6}'} X(3,1).
\end{equation*}

It follows from Gompf's theorem in \cite{Go} that $Z(3)$ is symplectic. 

\begin{lemma}\label{lemma:pi_1(X(3))=1} 
$Z(3)$ is simply-connected. 
\end{lemma}

\begin{proof}

By applying the Seifert-Van Kampen theorem, we see that

\begin{equation*}
\pi_{1}(Z(3)) \:=\: \frac{\pi_{1}(S\#\CPb\setminus \nu \widetilde{R})\ast \pi_{1}(X(3,1) \setminus \nu \Sigma'_{6})}{\langle a_{1} = a_{1}',\, b_{1} = b_{1}',\, \cdots,\, a_{6} = a_{6}',\, b_{6} = b_{6}',\, \mu = \mu' = 1 \rangle}.
\end{equation*}

\noindent where $a_i$, $b_i$, and $a_i'$, $b_i'$ (for $i= 1, \cdots, 6$) denote the standard generators of the fundamental group of the genus $6$ Riemann surfaces  $\widetilde{R}$ and $\Sigma'_{6}$ in $S\#\CPb$ and in $X(3,1)$, and $\mu$ and $\mu'$ denote their meridians in $S\#\CPb \setminus \nu \widetilde{R}$ and in $X(3,1) \setminus \nu \Sigma'_{6}$ respectively. Using the Proposition~\ref{prop1} (iii), and the facts that the normal circle $\mu=\{{\rm pt}\}\times S^1$ of $\widetilde{R}$ in $\pi_{1}(S\#\CPb \setminus \nu (\widetilde{R}))$ and the loops $a_{1}'$, $b_{1}'$, $\cdots$, $a_{6}'$, $b_{6}'$ in $\pi_{1}(X(3,1) \setminus \nu (\Sigma'_{6}))$ are all trivial, we see that the fundamental group of $Z(3)$ is the trivial group. 

\end{proof}

\begin{lemma} 
$e(Z(3))= 52$, $\sigma (Z(3)) = 0$, $c_1^{2}(Z(3)) = 104$, and\/ $\chi(Z(3)) = 13$.
\end{lemma}

\begin{proof} By applying the formulas~\ref{eq: invariant I} and \ref{eq: invariants II}, we have \/ $e(Z(3))=e(S\#\CPb)+e(X(3,1))+4(6-1)$, 
$\sigma (Z(3)) = \sigma(S\#\CPb) + \sigma(X(3,1))$, 
$c_1^{2}(Z(3)) = c_1^{2}(S\#\CPb) + c_1^{2}(X(3,1)) + 8(6-1)$, and 
$\chi(Z(3)) = \chi(S\#\CPb) + \chi(X(3,1)) + (6-1)$.  
Since $e(X(3,1))=16$, 
$\sigma(X(3,1))=-4$, 
 $c_1^{2}(X(3,1)=20$, 
$\chi(X(3,1))=3$, 
$e(S\#\CPb)=16$, 
$\sigma(S\#\CPb) = 4$, 
$c_1^{2}(S\#\CPb) = 44$, and 
$\chi(S\#\CPb) = 5$, 
the proof of lemma follows.

\end{proof}

Using Freedman's classification theorem for simply-connected 4-manifolds \cite{F}, the lemma above and the fact that $S\#\CPb$\/ contains genus two surface of self-intersection $-1$ disjoint from $\widetilde{R}$, we conclude that $Z(3)$\/ is homeomorphic to $(2n-1)\CP\#(2n-1)\CPb$ for $n = 13$. Since $Z(3)$ is symplectic, by Taubes's theorem \cite{Ta}) $Z(3)$ has non-trivial Seiberg-Witten invariant. Next, using the connected sum theorem for the Seiberg-Witten invariant, we deduce that the Seiberg-Witten invariant of $25\CP\#25\CPb$ is trivial. Since the Seiberg-Witten invariant is a diffeomorphism invariant, $Z(3)$\/ is not diffeomorphic to $25\CP\#25\CPb$. Furthermore, $Z(3)$\/ is a minimal symplectic $4$-manifold by Usher's Minimality Theorem \cite{U}. Since symplectic minimality implies smooth minimality (cf.\ \cite{Li}), $Z(3)$\/ is also smoothly minimal, and thus is smoothly irreducible.

To produce an infinite family of exotic $25\CP\#25\CPb$'s, we replace the building block $Y_{6}(1,1)$\/ used in our construction of $X(3,1)$ above with $Y_{6}(1,m)$ (see Section~\ref{sbb}, page 14), where $|m| > 1$. Let us denote the resulting smooth $4$-manifold as $Z(3,m)$\/. In the presentation of the fundamental group, the above surgery amounts replacing a single relation $[c^{-1},b_{n}] =d$ in $\pi_{1}(X(3,1))$, corresponding to the Luttinger surgery $(a_{n}'' \times d', d', 1)$, with ${[c^{-1},b_{n}]}^{-m}=d$. Notice that changing this relation has no affect on our proof of $\pi_{1}(Z(3)) = 1$; all the fundamental group calculations follow the same lines of arguments, and thus $\pi_{1}(Z(3,m))$ is trivial group.

Let us denote by $Z(3)_{0}$ the symplectic $4$-manifold obtained by performing the following Luttinger surgery on: $(a_{n}'' \times d', d', 0/1)$ instead of $(a_{n}'' \times d', d', 1)$ in the construction of $Z(3)$. It is easy to check that $\pi_{1}(Z(3)_{0}) = \mathbb{Z}$ and the canonical class of $Z(3)_{0}$ is given by the formula $K_{Z(3)_{0}} = K_{S\#\CPb} + 2[\Sigma_6]+ \sum_{j=1}^{4}[\bar{R_j}] + \Sigma_{6}' + \widetilde{R} + \dots $, where $\bar{R_j}$ are tori of self-intersection $-1$. Moreover, the Seiberg-Witten invariants of the basic class $\beta_{m}$ of $Z(3,m)$ corresponding to the canonical class $K_{Z(3)_{0}}$ evaluates as $SW_{Z(3)}(\beta_{m}) = SW_{Z(3)}(K_{Z(3)}) + (m-1)SW_{Z(3)_{0}} (K_{Z(3)_{0}}) = 1 + (m-1) = m$. Thus, we conclude that $Z(3,m)$ is nonsymplectic for any $m \geq 2$. 

Alternatively, we can use the rim tori that were constructed in the Remark~\ref{rem1}. Notice that these tori are Lagrangian, but we can perturb the symplectic form so that one of these tori, say $T$ becomes symplectic. Moreover, $\pi_1(Z(3)\setminus T) = 1$, which follows from the Van Kampen's Theorem using the facts that $\pi_1(Z(3)) = 1$ and the rim torus has nullhomotopic meridian. Hence, we have a symplectic torus $T$ in $Z(3)$ of self-intersection $0$ such that $\pi_1(Z(3)\setminus T) = 1$. By performing a knot surgery on $T$, inside $Z(3)$, we acquire an irreducible 4-manifold $Z(3)_K$ that is homeomorphic to $Z(3)$. By varying our choice of the knot $K$, we can realize infinitely many pairwise non-diffeomorphic 4-manifolds, either symplectic or nonsymplectic.

Furthemore, by applying Theorem~\ref{thm:gt}, and then Theorem~\ref{thm:gt1} to symplectic $4$-manifold $Z(3)$, we obtain infinitely many minimal symplectic\/ $4$-manifolds and infinitely many non-symplectic $4$-manifolds that is homeomorphic but not diffeomorphic to\/ $(2n-1)\CP\#(2n-2)\CPb$ for any integer $n \geq 14$. This concludes the proof of our theorem.

Next, we prove the following theorem which considers the case $n=12$. Since the proof is similar to the proof of previous theorem, we omit some details

\begin{theorem}\label{thm:main.2}
Let $M$\/ be \/$23\CP\#23\CPb$. There exists an irreducible symplectic\/ $4$-manifold and an infinite family of pairwise non-diffemorphic irreducible non-symplectic $4$-manifolds that all of which are homeomorphic to\/ $M$. 
\end{theorem}

Our first building block again will be the complex surface $S\#\CPb$ along with the genus $6$ complex submanifold $\widetilde{R} \subset S\#\CPb$ that was constructed in Section~\ref{sbb}. Let us endow $S\#\CPb$ with the symplectic structure induced from the K\"{a}hler structure. Our second building block will be obtained from the symplectic $4$-manifold $X_{2,4}$ via two blow-ups. Recall from Theorem~\ref{thm2} that $X_{2,4}$ contains symplectic surface $\Sigma_2$ with self intersection $0$ and two genus $2$ surfaces, say  $S_{1}$ and $S_{2}$, with self intersections $-1$. Moreover, $S_{1}$ and $S_{2}$ intersect with $\Sigma_2$ positively and transversally. By symplectically resolving the intersections of $\Sigma_2$ with $S_1$ and $S_2$, we obtain the genus six symplectic surface $\Sigma_{6}'$ of square $+2$ in $X_{2,4}$. We symplectically blow up $\Sigma_{6}'$ at two points to obtain a symplectic surface $\Sigma_{6}''$ of self intersection $0$ in $X_{2,4} \# 2\overline{\mathbb{CP}}^{2}$ (see Figure~\ref{dt}).

We denote by $Z(2)$ the symplectic $4$-manifold obtained by forming the symplectic connected sum of $S\#\CPb$ and $X_{2,4} \# 2\overline{\mathbb{CP}}^{2}$\/ along the surfaces $\widetilde{R}$ and $\Sigma_{6}"$.

\begin{equation*}
Z(2) = (S \# \overline{\mathbb{CP}}^{2}) \#_{\widetilde{R} = \Sigma_{6}''} X_{2,4} \# 2\overline{\mathbb{CP}}^{2}
\end{equation*}

It follows from Gompf's theorem in \cite{Go} that $Z(2)$ is symplectic. 

\begin{lemma}\label{lemma:pi_1(X(3))=1} 
$Z(2)$ is simply-connected. 
\end{lemma}

\begin{proof}

This follows from Van Kampen's Theorem. Notice that we have

\begin{equation*}
\pi_{1}(Z(2)) \:=\: \frac{\pi_{1}(S\#\CPb\setminus \nu \widetilde{R})\ast \pi_{1}(X_{2,4} \# 2\overline{\mathbb{CP}}^{2} \setminus \nu \Sigma_{6}'')}{\langle a_{1} = a_{1}'',\, b_{1} = b_{1}'',\, \cdots,\, a_{6} = a_{6}'',\, b_{6} = b_{6}'',\, \mu = \mu'' = 1 \rangle}.
\end{equation*}

\noindent where $a_i$, $b_i$, and $a_i''$, $b_i''$ (for $i= 1, 2, 3$) denote the standard generators of the fundamental group of the genus $6$ Riemann surfaces  $\widetilde{R}$ and $\Sigma_{6}''$ in $S\#\CPb$ and in $X_{2,4} \# 2\overline{\mathbb{CP}}^{2}$, and $\mu$ and $\mu''$ denote their meridians respectively.

\noindent By applying the Proposition~\ref{prop1} (iii), and the facts that the normal circle $\mu$ of $\widetilde{R}$ in $\pi_{1}(S\#\CPb \setminus \nu \widetilde{R})$ and the loops $a_{1}''$, $b_{1}''$, $\cdots$, $a_{6}''$, $b_{6}''$, and $\mu''$ in $\pi_{1}(X_{2,4} \# 2\overline{\mathbb{CP}}^{2} \setminus \nu \Sigma_{6}'')$ are all trivial, we conclude that the fundamental group of $Z(2)$ is trivial. 

\end{proof}

\begin{lemma} 
$e(Z(2))= 48$, $\sigma (Z(2)) = 0$, $c_1^{2}(Z(2)) = 96$, and\/ $\chi(Z(2)) = 12$.
\end{lemma}

\begin{proof} Using the formulas~\ref{eq: invariant I} and \ref{eq: invariants II}, we have \/ $e(Z(2))=e(S\#\CPb)+e(X_{2,4} \# 2\overline{\mathbb{CP}}^{2})+4(6-1)$, $\sigma (Z(2)) = \sigma(S\#\CPb) + \sigma(X_{2,4} \# 2\overline{\mathbb{CP}}^{2})$, 
$c_1^{2}(Z(2)) = c_1^{2}(S\#\CPb) + c_1^{2}(X_{2,4} \# 2\overline{\mathbb{CP}}^{2}) + 8(6-1)$, and 
$\chi(Z(2)) = \chi(S\#\CPb) + \chi(X_{2,4} \# 2\overline{\mathbb{CP}}^{2}) + (6-1)$.  
Since $e(X_{2,4} \# 2\overline{\mathbb{CP}}^{2})=12$, 
$\sigma(X_{2,4} \# 2\overline{\mathbb{CP}}^{2})=-4$, 
 $c_1^{2}(X_{2,4} \# 2\overline{\mathbb{CP}}^{2})=16$, 
$\chi(X_{2,4} \# 2\overline{\mathbb{CP}}^{2})=2$, 
$e(S\#\CPb)=16$, 
$\sigma(S\#\CPb) = 4$, 
$c_1^{2}(S\#\CPb) = 44$, and 
$\chi(S\#\CPb) = 5$, 
the proof of lemma readily follows.

\end{proof}

Now by the lemmas above, Freedman's classification theorem for simply-connected 4-manifolds \cite{F}, and the fact that $Z(2)$\/ contains $-1$ genus two surface resulting from internal sum, we see that $Z(2)$\/ is homeomorphic to $23\CP\#23\CPb$. Since $Z(2)$ is symplectic and has non-trivial Seiberg-Witten invariants, $Z(2)$\/ is an exotic copy of $23\CP\#23\CPb$. To produce an infinite family of exotic $23\CP\#23\CPb$'s, we need to replace the building block $Y_{2}(1,1)$\/ used in our construction of 
$X_{2,4}$ above with $Y_{2}(1,m)$, where $|m| > 1$. The proof of the rest of the theorem is identical to that of
Theorem~\ref{thm:main.1}, and therefore we omit the details.

\section{Construction of exotic $4$-manifolds with positive signature}\label{conII}

In this section, we will construct the families of simply connected non-spin symplectic and smooth 4-manifolds with positive signature and small $\chi$ . Our construction will prove the second main theorem (Theorem~\ref{thm:main2}) of this paper stated in the introduction. We will first prove the Theorem~\ref{thm:main2} in special cases of (i)-(iii), and then derive the general cases using the Theorems~\ref{thm:gt}, ~\ref{thm:gt1}, and Corollary~\ref{cor1}. The  generalizations of the results of this section for other fundamental groups and higher values of $\chi$ is considered in \cite{SS}.

\subsection{Signature Equal to 1 Case}

Let us begin with the construction of an exotic copy of $27\mathbb{CP}^{2} \# 26\overline{\mathbb{CP}}^{2}$, which exemplifies the signature equal to $1$ case (i.e. the case (i) of Theorem~\ref{thm:main2}). 

Our first building block is the complex surface $S\#\overline{\mathbb{CP}}^{2}$ along with the genus $6$ symplectic surface $\widetilde{R}$ constructed in Section~\ref{HCB}. The second building block is obtained from the symplectic $4$-manifold $X_{4,6}$, in the notation of Theorem~\ref{thm2}. We will use the fact that $X_{4,6}$ contains a symplectic genus two surface $\Sigma_2$ with self-intersection $0$ and two genus $4$ symplectic surfaces with self intersections $-1$ intersecting $\Sigma_2$ positively and transversally.  For the convenience of the reader, we briefly review the construction of $X_{4,6}$ (see \cite{ABBKP} for the details). Take a copy of $\mathbb{T}^2 \times \{pt\}$ and $\{pt\} \times \mathbb{T}^2$ in $\mathbb{T}^2 \times \mathbb{T}^2$ equipped with the product symplectic form, and symplectically resolve the intersection point of these dual symplectic tori. The resolution produces symplectic genus two surface of self intersection $+2$ in $\mathbb{T}^2 \times \mathbb{T}^2$. By symplectically blowing up this surface twice, in $\mathbb{T}^4 \# 2 \overline{\mathbb{CP}}^{2}$, we obtain a symplectic genus 2 surface $\Sigma_2$ with self-intersection $0$, with two $-1$ spheres (i.e. the exceptional spheres resulting from the blow-ups) intersecting it positively and transversally. Next, we form the symplectic connected sum of $\mathbb{T}^4\#2\overline{\mathbb{CP}}^{2}$ with $\Sigma_2 \times \Sigma_4$ along the genus two surfaces  $\Sigma_2$ and $\Sigma_2 \times \{pt\}$. By performing the sequence of appropriate $\pm 1$ Luttinger surgeries on $(\mathbb{T}^4 \# 2 \overline{\mathbb{CP}}^{2}) \#_{\Sigma_2 = \Sigma_2 \times \{pt\}} (\Sigma_2 \times \Sigma_4)$, we obtain the symplectic $4$-manifold $X_{4,6}$ constructed in \cite{ABBKP} (see Theorem 5.1, page 14), which is an exotic copy of $7\mathbb{CP}^{2} \# 9\overline{\mathbb{CP}}^{2}$
. It can be seen from the construction that, $X_{4,6}$ contains symplectic surface $\Sigma_2$ with self intersection $0$ and two genus $4$ surfaces  $S_{1}$ and $S_{2}$ with self intersections $-1$ which have positive and transverse intersections with $\Sigma_2$. Notice that the surfaces $S_{1}$ and $S_{2}$ result from the internal sum of the punctured exceptional spheres 
in $\mathbb{T}^4 \# 2\overline{\mathbb{CP}}^{2} \setminus \nu (\Sigma_2)$ and the punctured genus four surfaces in $\Sigma_2 \times \Sigma_4 \setminus \nu (\Sigma_2 \times \{pt\})$ (see the Figure~\ref{dt}). Moreover,  $X_{4,6}$ contains a pair of disjoint Lagrangian tori $T_{1}$ and $T_{2}$ with the same properties as assumed in the statement of the Corollary~\ref{cor1}. Notice that these Lagrangian tori descend from $\Sigma_2 \times \Sigma_4$, and survive in  $X_{4,6}$ after symplectic connected sum and the Luttinger surgeries. This is because there are at least two pairs of Lagrangian tori in $\Sigma_2 \times \Sigma_4$ that were away from the standard symplectic surfaces $\Sigma_2 \times \{pt\}$ and $ \{pt\} \times \Sigma_4$, and the Lagrangian tori that were used for Luttinger surgeries (for an explanation, see subsection~\ref{L}, page 13). Also, the fact that $\pi_{1}(X_{4,6}\setminus(T_{1}\cup T_{2})) = 1$ is explained in details in \cite{AP3} (see proof of Theorem 8, page 272).

Next, we symplectically resolve the intersection of $\Sigma_2$ and one of the genus $4$ surfaces, say $S_1$, in $X_{4,6}$. This produces the genus six surface $\Sigma_6'$ of square $+1$ intersecting the other genus $4$  surface $S_{2}$ with self-intersection $-1$. We blow up $\Sigma_6'$ at a point to obtain a symplectic surface $\Sigma_6$ of self intersection $0$ in $X_{4,6} \# \overline{\mathbb{CP}}^{2}$ (see Figure~\ref{dt}).

Since each of the two symplectic building blocks $S \#  \overline{\mathbb{CP}}^{2}$ and $X_{4,6} \# \overline{\mathbb{CP}}^{2}$ contain symplectic genus $6$ surfaces of self intersection $0$,  we can form their symplectic connected sum along these surfaces $\widetilde{R}$ and $\Sigma_6$. Let
\begin{equation*}
M_{1,4} = (S \#  \overline{\mathbb{CP}}^{2}) \#_{\widetilde{R} = \Sigma_6} (X_{4,6} \# \overline{\mathbb{CP}}^{2}).
\end{equation*}

\begin{figure}[ht]
\begin{center}
\includegraphics[scale=.89]{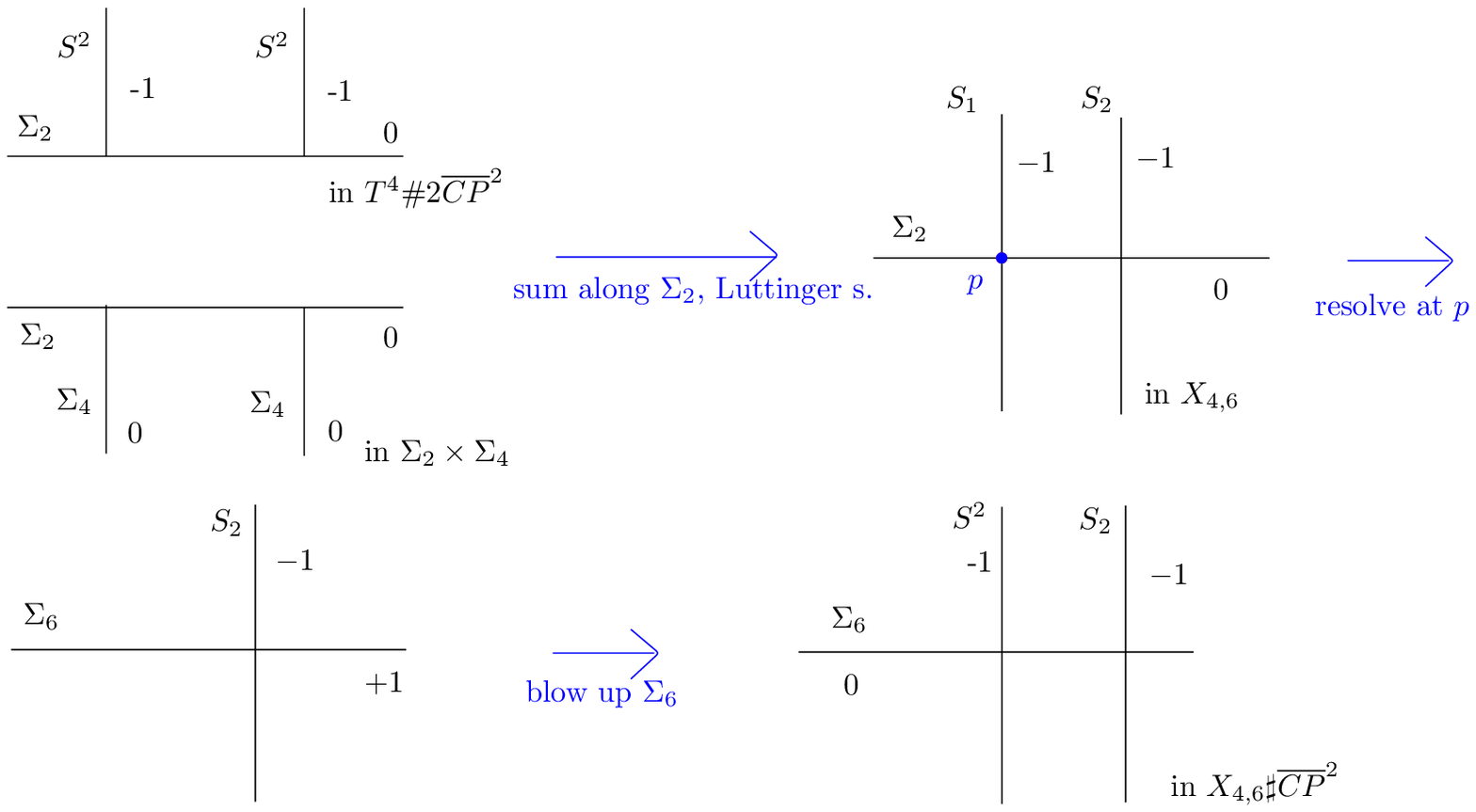}
\caption{}
\label{dt}
\end{center}
\end{figure}

\begin{lemma}
$e(M_{1,4}) = 55$, $\sigma(M_{1,4}) = 1$, $c_1^2(M_{1,4}) = 113$, $\chi(M_{1,4}) = 14$.
\end{lemma}

\begin{proof}  We will use the topological invariants of $X_{4,6}$ and $S\#  \overline{\mathbb{CP}}^{2}$ to compute the topological invariants of $M_{1,4}$. Since $e(S) = 15$, $\sigma(S) = 5$, $c_1^2(S) = 45$, $\chi(S) = 5$, we have $e(S\#  \overline{\mathbb{CP}}^{2}) = 16$, $\sigma(S\#  \overline{\mathbb{CP}}^{2}) = 4$, $c_1^2(S\#  \overline{\mathbb{CP}}^{2}) = 44$, $\chi(S\#  \overline{\mathbb{CP}}^{2}) = 5$. Also, by Theorem~\ref{thm2}, we have $e(X_{4,6}) = 18$, $\sigma(X_{4,6}) = -2$, $c_1^2(X_{4,6}) = 30$, $\chi(X_{4,6}) = 4$. Thus, we have  $e(X_{4,6} \# \overline{\mathbb{CP}}^{2}) = 19$, $\sigma(X_{4,6} \# \overline{\mathbb{CP}}^{2}) = -3$, $c_1^2(X_{4,6} \# \overline{\mathbb{CP}}^{2}) = 29$, $\chi(X_{4,6} \# \overline{\mathbb{CP}}^{2}) = 4$. Now using the formulas~\ref{eq: invariant I} and \ref{eq: invariants II} for symplectic connected sum, we compute the topological invariants of $M_{1,4}$ as given above.
\end{proof}

Similary as in the signature zero case in Section~\ref{conI},  we show that $M_{1,4}$ is symplectic and simply connected, using Gompf's Theorem~\ref{thm:symsum} and Van Kampen's Theorem respectively. Using the same lines of arguments as in Section~\ref{conI}, we see that $M_{1,4}$ is an exotic copy of $27\mathbb{CP}^{2} \# 26 \overline{\mathbb{CP}}^{2}$. Moreover, as was explained above, $M_{1,4}$ contains a pair of disjoint Lagrangian tori $T_{1}$ and  $T_{2}$ of self-intersection $0$ such that $\pi_{1}(M_{1,4}\setminus(T_{1}\cup T_{2})) = 1$. We can perturb the symplectic form on $M_{1,4}$ in such a way that one of the tori, say $T_{1}$, becomes symplectically embedded. The reader is refered to Lemma 1.6 \cite{Go} for the existence of such perturbation. We perform a knot surgery, (using a knot $K$ with non-trivial Alexander polynomial) on $M_{1,4}$ along $T_{1}$ to obtain irreducible 4-manifold $(M_{1,4})_K$ that is homeomorphic but not diffemorphic to $M_{1,4}$. By varying our choice of the knot $K$, we can realize infinitely many pairwise non-diffeomorphic 4-manifolds, either symplectic or nonsymplectic (see Theorem~\ref{thm:gt1}). Finally, by applying Theorems~\ref{thm:gt}, \ref{thm:gt1}, and Corollary~\ref{cor1}, we also obtain infinitely many irreducible symplectic and infinitely many irreducible non-symplectic $4$-manifolds that is homeomorphic but not diffeomorphic to\/ $(2n-1)\CP\#(2n-2)\CPb$ for any integer $n \geq 15$.   

\subsection{Signature Equal to 2 Case} The construction in this case is similar to that of $\sigma = 1$ case above, therefore we will omit some of the already familiar details. We will first construct an exotic copy of $25\mathbb{CP}^{2} \# 23 \overline{\mathbb{CP}}^{2}$, and use the Theorems~\ref{thm:gt} and~\ref{thm:gt1} and Corollary~\ref{cor1} to deduce the general case. Our first building block again is $S \#\overline{\mathbb{CP}}^{2}$, containing genus $6$ surface $\widetilde{R}$ of square $0$. To obtain the second symplectic building block, we form the symplectic connected sum of $\mathbb{T}^4 \# 2\overline{\mathbb{CP}}^{2}$ with $\Sigma_2 \times \Sigma_5$ along the genus two surfaces  $\Sigma_2$ and $\Sigma_2 \times \{pt\}$. Let

\begin{equation*}
X_{5,7} = (\mathbb{T}^4 \# 2\overline{\mathbb{CP}}^{2}) \#_{\Sigma_2 = \Sigma_2 \times \{pt\}} (\Sigma_2 \times \Sigma_5).
\end{equation*}

\begin{figure}[ht]
\begin{center}
\includegraphics[scale=.89]{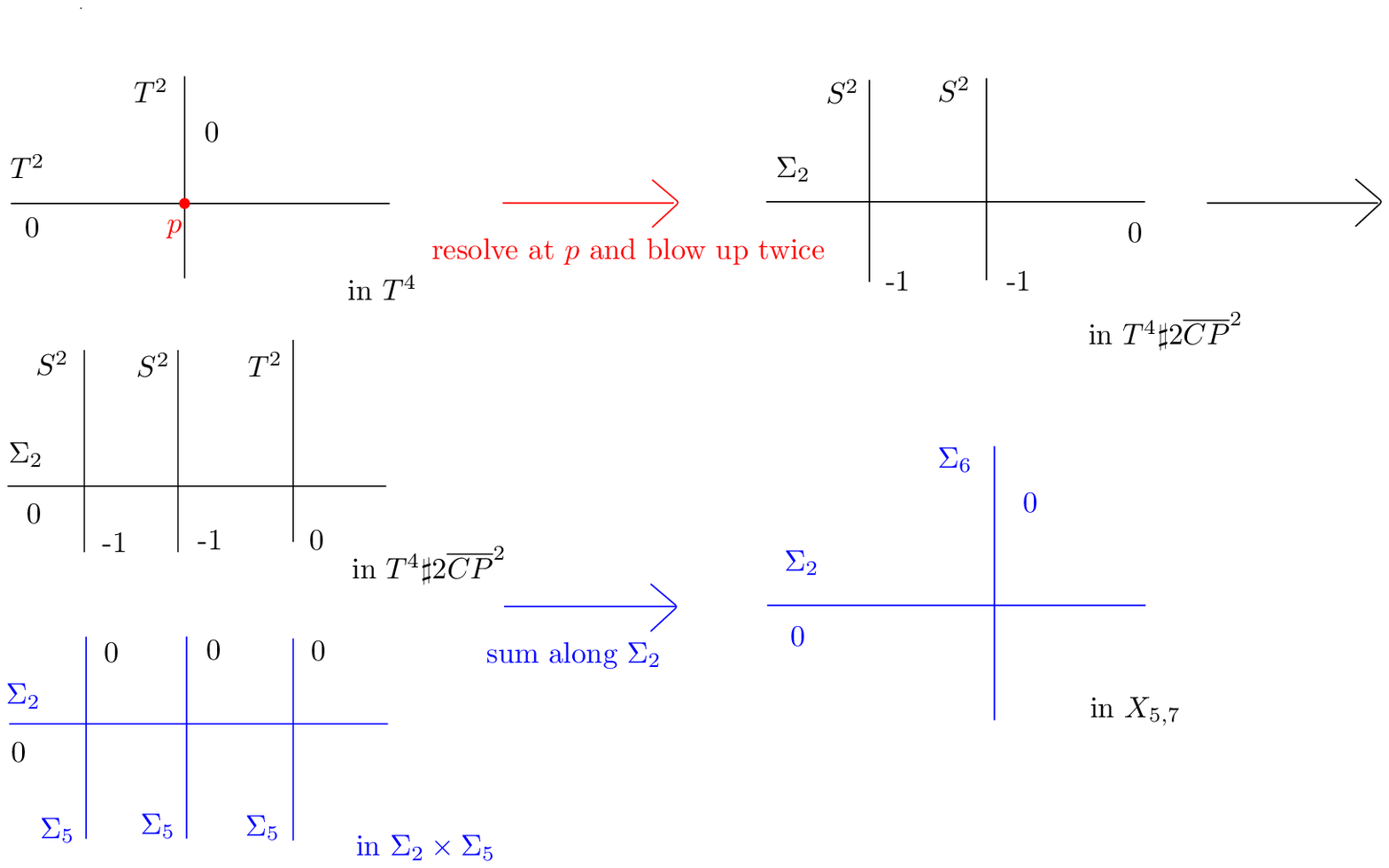}
\caption{}
\label{dtt}
\end{center}
\end{figure}

It was shown in \cite{ABBKP} (see Theorem 5.1, page 14), that $X_{5,7}$, which is an exotic copy of $9\CP\#11\CPb$. Using the Figure~\ref{dtt}, it is easy to see that $X_{5,7}$ contains a symplectic genus $6$ surface $\Sigma_6$ of square $0$ resulting from the internal sum of a punctured genus one surface in $\mathbb{T}^4 \# 2\overline{\mathbb{CP}}^{2} \setminus \nu (\Sigma_2)$ and a punctured genus five surface $\Sigma_5$ in $ \Sigma_2 \times \Sigma_5 \setminus \nu (\Sigma_2 \times \{pt\})$. Next, we form the symplectic connected sum of $S \#\overline{\mathbb{CP}}^{2}$ and $X_{5,7}$ along the genus six surfaces $\widetilde{R}$ and $\Sigma_6$  

\begin{equation*}   
M_{2,5}= (S \#  \overline{\mathbb{CP}}^{2}) \#_{\widetilde{R} = \Sigma_6} X_{5,7}.
\end{equation*} 
along the copies of $\Sigma_6$ in both of the 4-manifolds. It is easy to check that the following lemma holds

\begin{lemma}
$e(M_{2,5}) = 50$, $\sigma(M_{2,5}) = 2$, $c_1^2(M_{2,5}) = 106$, $\chi(M_{2,5}) = 13$.
\end{lemma}

We conclude as above that $M_{2,5}$ is symplectic and simply connected and an exotic copy of $25\mathbb{CP}^{2} \# 23 \overline{\mathbb{CP}}^{2}$. Once again, by applying Theorems~\ref{thm:gt} and \ref{thm:gt1}, and Corollary~\ref{cor1}, we obtain infinitely many minimal symplectic\/ $4$-manifolds and an infinitely many non-symplectic $4$-manifolds that is homeomorphic but not diffeomorphic to\/ $(2n-1)\CP\#(2n-3)\CPb$ for any integer $n \geq 13$.

\subsection{Signature Equal to 3 Case}

In what follows, we will construct simply connected non-spin irreducible symplectic and smooth 4-manifolds with signature $3$. We will first consider a special case in which our construction yields infinitely many exotic copies of $29\mathbb{CP}^{2} \# 26 \overline{\mathbb{CP}}^{2}$. The general case again will be proved by appealing to Theorems~\ref{thm:gt}, \ref{thm:gt1}, and Corollary~\ref{cor1}. 

The first building block is again $S \# \overline{\mathbb{CP}}^{2}$ and the second building block is the symplectic $4$-manifold $X_{5,6}$, an exotic $9\mathbb{CP}^{2} \# 10\overline{\mathbb{CP}}^{2}$ constructed in \cite{AP2}. Let us recall the construction of exotic copy of $9\mathbb{CP}^{2} \# 10\overline{\mathbb{CP}}^{2}$ from \cite{AP2}. We take a copy of $\mathbb{T}^2 \times \{pt\}$ and the braided torus $T_{\beta}$ representing the homology class  $2[\{pt\} \times \mathbb{T}^2]$ in $\mathbb{T}^2 \times \mathbb{T}^2$ (see ~\cite{AP2}, page 4 for the construction of $T_{\beta}$). The tori $\mathbb{T}^2 \times \{pt\}$ and $T_{\beta}$ intersect at two points. We symplectically blow up one of these intersection points, and symplectically resolve the other intersection point to obtain the symplectic genus two surface of self intersection $0$ in  $\mathbb{T}^4 \# \overline{\mathbb{CP}}^{2}$ (see ~\cite{AP2}, pages 3-4).  The symplectic genus $2$ surface $\Sigma_2$ has a dual symplectic torus $\mathbb{T}^2$ of self intersections zero intersecting $\Sigma_2$ positively and transversally at one point. We form the symplectic connected sum of $\mathbb{T}^4 \# \overline{\mathbb{CP}}^{2}$ with $\Sigma_2 \times \Sigma_5$ along the genus two surfaces  $\Sigma_2$ and $\Sigma_2 \times \{pt\}$. By performing the sequence of appropriate $\pm 1$ Luttinger surgeries on $(\mathbb{T}^4 \#  \overline{\mathbb{CP}}^{2}) \#_{\Sigma_2 = \Sigma_2 \times \{pt\}} (\Sigma_2 \times \Sigma_5)$, we obtain the symplectic $4$-manifold $X_{5,6}$ constructed in \cite{AP2}. It can be seen from the construction that, $X_{5,6}$ contains symplectic surface $\Sigma_6$ with self intersection $0$, resulting from the internal sum of the punctured torus
in $\mathbb{T}^4 \# \overline{\mathbb{CP}}^{2} \setminus \nu (\Sigma_2)$ and the punctured genus five surfaces in $\Sigma_2 \times \Sigma_5 \setminus \nu (\Sigma_2 \times \{pt\})$ (see the Figure~\ref{dttt}). Furthemore, $X_{5,6}$ contains a pair of disjoint Lagrangian tori $T_{1}$ and $T_{2}$ with the properties required by Corollary~\ref{cor1}. These Lagrangian tori descend from $\Sigma_2 \times \Sigma_5$ and survive in  $X_{5,6}$ after symplectic connected sum and the Luttinger surgeries.

\begin{figure}
\begin{center}
\scalebox{0.89}{\includegraphics{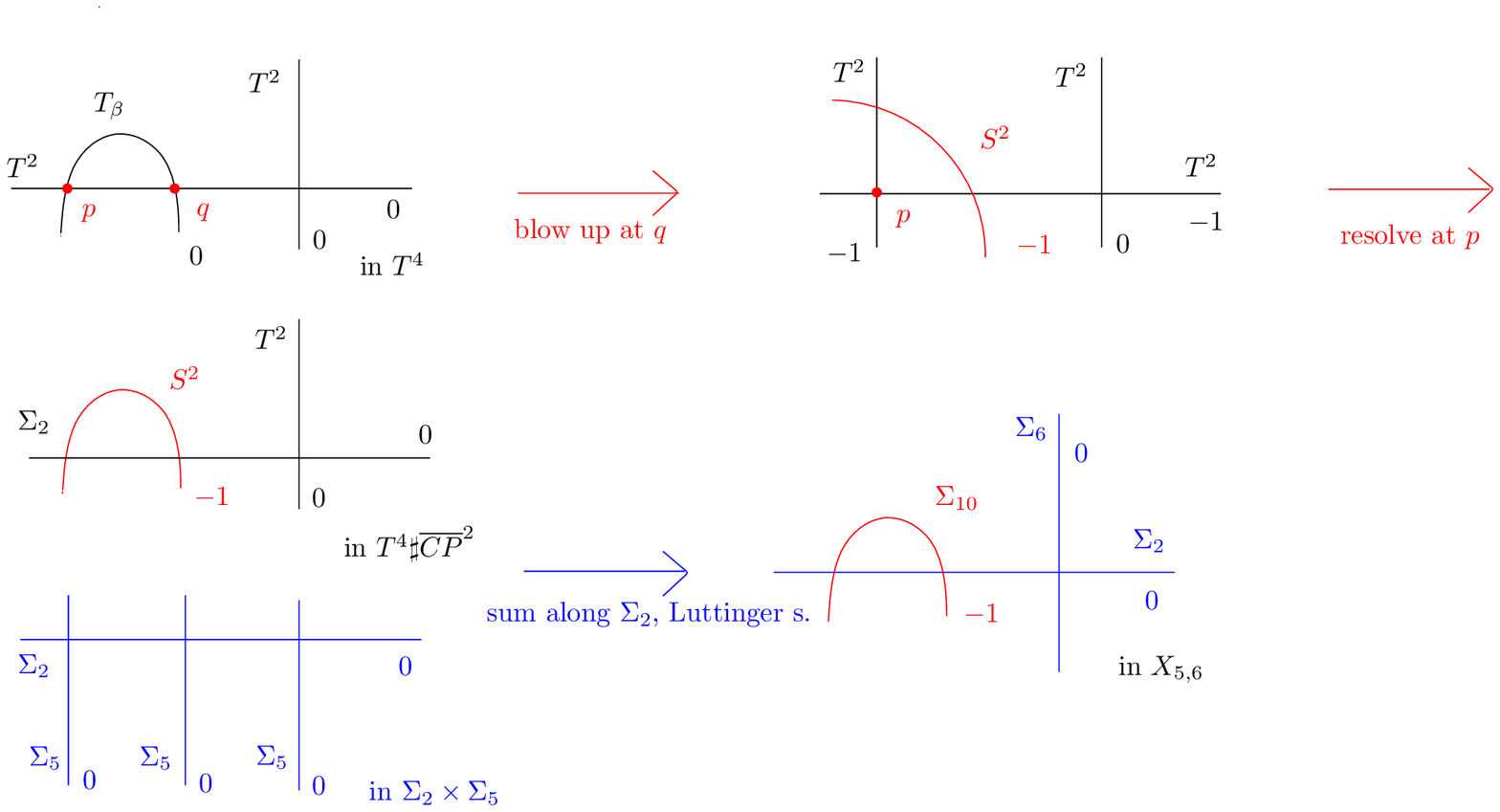}}
\caption{}
\label{dttt}
\end{center}
\end{figure}

As in the signature $1$ and $2$ cases above, we will form the symplectic connected sum along genus $6$ surfaces. Let

\begin{equation*}
M_{3,5} = (S \#  \overline{\mathbb{CP}}^{2}) \#_{\widetilde{R} = \Sigma_6} (X_{5,6}).
\end{equation*} 

\begin{lemma}
$e(M_{3,5}) = 57$, $\sigma(M_{3,5}) = 3$, $c_1^2(M_{3,5}) = 123$, $\chi(M_{3,5}) = 15$.
\end{lemma}

\begin{proof} 
Firstly, we compute the topological invariants $X_{5,6}$.  Notice that $e(\mathbb{T}^4 \# \overline{\mathbb{CP}}^{2}) = 1$, $\sigma(\mathbb{T}^4 \# \overline{\mathbb{CP}}^{2}) = -1$, $c_1^2(\mathbb{T}^4 \# \overline{\mathbb{CP}}^{2}) = -1$, $\chi(\mathbb{T}^4 \# \overline{\mathbb{CP}}^{2}) = 0$. For $\Sigma_2 \times \Sigma_5$, we have $e(\Sigma_2 \times \Sigma_5) = 16$, $\sigma(\Sigma_2 \times \Sigma_5)=0$, $c_1^2(\Sigma_2 \times \Sigma_5)=32$ and $\chi(\Sigma_2 \times \Sigma_5)=4$. Therefore, for the symplectic connected sum manifold $X_{5,6}$, we have $e(X_{5,6})= 21$, $\sigma(X_{5,6})=-1$, $c_1^2(X_{5,6})=39$ and $\chi(X_{5,6})=5$. With the invariants of $S \#  \overline{\mathbb{CP}}^{2}$ and $X_{5,6}$ at hand, we compute the topological invariants of $M_{3,5}$ as above using the formulas~\ref{eq: invariant I} and \ref{eq: invariants II}.

\end{proof}

Following the arguments as in the proof of Theorem~\ref{thm:main.1}, we see that $M_{3,5}$ is an exotic copy of  $29\mathbb{CP}^{2} \# 26 \overline{\mathbb{CP}}^{2}$, which is also smoothly minimal. Once again, by applying Theorems~\ref{thm:gt}, \ref{thm:gt1}, and Corollary~\ref{cor1}, we obtain infinitely many minimal symplectic\/ $4$-manifolds and an infinite family of non-symplectic $4$-manifolds that is homeomorphic but not diffeomorphic to\/ $(2n-1)\CP\#(2n-4)\CPb$ for any integer $n \geq 15$.


\subsection{Signature greater than 3 Case}\label{sig4} In what follows, we discuss how to extend the constructions given in Theorem~\ref{thm:main2} to obtain the simply connected non-spin irreducible symplectic $4$-manifolds with $\sigma > 3$. Our motivation for constructing such examples comes from the article \cite{AHP}, where the geography of simply connected non-spin 4-manifolds with positive signature are studied in details. We will make use of a very recent construction of Catanese and Detweiller in \cite{CD} (see Section 4), which generalizes the complex surfaces of Hirzebruch and Bauer-Catanese with invariants $c_1^2 = 9\chi_h = 45$ that we employed in the proof of Theorem~\ref{thm:main2}. Let $n > 1$ be any positive integer relatively prime with $6$. In \cite{CD}, using $(\mathbb{Z}/n\mathbb{Z})^{2}$ Galois coverings of the rational surface, an infinite family of complex surfaces $S(n)$ of general type with  ${c_1}^2 (S(n))= 5(n-2)^{2}$, $c_2(S(n))= 2n^{2} - 10n + 15$, $\sigma(S(n)) = 1/3(n^2 -10)$ and irregularity $q = (n-1)/2$ are constructed. The surfaces $S(n)$ admit a genus $n-1$ fibration over genus $k := (n-1)/2$ surface with three singular fibers, and each singular fiber consists of two smooth curves of genus $k$ intersecting transversally in exactly one point (see Proposition 29 in \cite{CD}, page 15). Notice that in the special case of $n=5$, the surface $S(5)$ is the complex surfaces of Hirzebruch and Bauer-Catanese.  Furthemore, the analog of Proposition~\ref{prop1} holds for $S(n)$, which show the existence of genus $3k$ symplectic surface $\widetilde{R_{n}}$ in $S(n)\#\CPb$ with self-intersection zero and with $\pi_1(\widetilde{R_{n}})\rightarrow \pi_1(S(n)\#\CP)$ being surjective. Using the symplectic sum of $S(n)\#\CPb$ (for $n > 5$) and the appropriate exotic symplectic $4$-manifolds constructed in Section~\ref{sbb} along the genus $3k$ surfaces, we obtain the symplectic $4$-manifolds with $\sigma > 3$. Since the proofs are similar to those already given in Theorems~\ref{thm:main1} and \ref{thm:main2}, we omit the details. We would like to remark that the examples discussed here significantly improves the bound $\lambda(\sigma)$ studied in \cite{AP3, AHP} for $\sigma \geq 0$. 

\begin{remark} In \cite{A1}, the first author has given a construction of an infinite family of fake rational homology $(2n-1)\CP\#(2n-1)\CPb$ for any integer $n \geq 3$, and the approach presented in \cite{A1} is promising in constructing the exotic smooth structures on $4$-manifolds with nonnegative signature and $\chi \geq 3$. We hope that using the building blocks discussed in this article and the ones studied in \cite{A1}, one can construct symplectic $4$-manifolds that is homeomorphic but not diffeomorphic to\/ $(2n-1)\CP\#(2n-1)\CPb$ for various $n$ with $3 \leq n \leq 11$. We will return to this problem in a follow up project. 
\end{remark}

\section*{Acknowledgments} The first author would like to thank Professor Fabrizio Catanese for a useful email exchange on surfaces in \cite{Main}. A. Akhmedov was partially supported by NSF grants DMS-1065955, DMS-1005741 and Sloan Research Fellowship. S. Sakall{\i} was partially supported by NSF grants DMS-1065955.


\begin{thebibliography}{9999}

\bibitem{A1} A. Akhmedov: {\em Note on new symplectic 4-manifolds with nonnegative signature}, arXiv 1207.1973v1, (2012).


\bibitem{A4} A. Akhmedov: {\em Small exotic 4-manifolds}, Algebr. Geom. Topol., {\bf 8} (2008), 1781--1794.

\bibitem{ABBKP} A. Akhmedov, S. Baldridge, I. Baykur, P. Kirk, B.~D. Park: {\em Simply connected minimal symplectic 4-manifolds with signature less than -1},  J. Eur. Math. Soc., {\bf 12} (1), (2010), 133--161.


\bibitem{AHP} A. Akhmedov, M. Hughes, and B.~D. Park: {\em Geography of simply-connected nonspin $4$-manifolds with positive signature}, Pacific J. Math., {\bf 262} (2), 2013, 257--282. 

\bibitem{AP1}  A. Akhmedov, B.~D. Park: {\em Exotic 4-manifolds with small Euler characteristic}, Invent. Math., {\bf 173} (2008), 209--223. 

\bibitem{AP2} A. Akhmedov, B.~D. Park: {\em Exotic smooth structures on small 4-manifolds with odd signatures}, Invent. Math., {\bf 181} (3), 2010, 577--603.

\bibitem{AP3} A. Akhmedov, B.~D. Park: {\em New symplectic $4$-manifolds with nonnegative signature}, J. G\"{o}kova Geom. Topol., {\bf 2} (2008), 1--13.

\bibitem{AP5} A. Akhmedov, B.~D. Park: Geography of Simply Connected Spin Symplectic 4-Manifolds, Math. Res. Letters, {\bf 17} (2010), no. 3, 483--492.

\bibitem{APU}  A. Akhmedov, B.~D. Park, G. Urzua: {\em Spin symplectic 4-manifolds near Bogomolov-Miyaoka-Yau line}, J. G\"{o}kova Geom. Topol., {\bf 4} (2010), 55--66.

\bibitem{AS} A. Akhmedov, N. Saglam: {\em New exotic 4-manifolds via Luttinger surgery on Lefschetz fibrations}, Internat. J. Math.  {\bf 26}  (2015),  no. 1. 

\bibitem{ADK}  D. Auroux, S. K. Donaldson, L. Katzarkov.: {\em Luttinger surgery along Lagrangian tori and non-isotopy for singular symplectic plane curves}, Math. Ann. \textbf{326} (2003), 185--203.

\bibitem{BHPV} W. P. Barth, K. Hulek, C. A. M. Peters, A. Van de Ven.: {\em Compact complex surfaces}, Springer-Verlag, Berlin Heidelberg, Second Enlarged Edition, 2004. 

\bibitem{Main} I. C. Bauer, F. Catanese.: {\em A Volume Maximizing Canonical Surface In 3-Space}, Comment. Math. Helv., {\bf 83}, 2008, 387--406.

\bibitem{CD}  F. Catanese, M. Dettweiler: {\em Vector Bundles on Curves Coming from Variation of Hodge Structures }, arXiv:1505.05064. 


\bibitem{Ch} Z. Chen:  {\em On the geography of surfaces (simply connected minimal surfaces with positive index)}, Math. Ann., {\bf 277} (1987), 141--164.

\bibitem{Hirze} F. Hirzebruch: {\em Arrangements of Lines and Algebraic Surfaces}, Arithmetic and geometry : papers dedicated to I.R. Shafarevich on the occasion of his sixtieth birthday /
TN: 982123, volume 2, 1983, 113--140.

\bibitem{FS5} R. Fintushel, R. J. Stern: {\em Surgery in cusp neighborhoods and the geography of irreducible 4-manifolds}, Invent. Math., {\bf 117} (1994), no. 3, 455--523.

\bibitem{FPS}  R. Fintushel, B. D. Park, R. J. Stern:  {\em Reverse engineering small 4-manifolds},  Algebr. Geom. Topol., \textbf{7} (2007), 2103--2116.

\bibitem{F} M. Freedman: {\em The topology of four-dimensional manifolds}, J. Diff. Geom., {\bf 17}, (1982), 357--453. 

\bibitem{Go}  R. Gompf: {\em A new construction of symplectic manifolds}, Annals of Math., {\bf 142} (1995), 527--595.

\bibitem{GS} R. Gompf, A. Stipsicz: {\em 4-Manifolds and Kirby Calculus}, Graduate Studies in Mathematics, vol. 20, 
Amer. Math. Soc., Providence, RI, 1999.

\bibitem{Gu} Y. Gurtas: {\em Positive Dehn twist expressions for some elements of finite order in the mapping class group II}, Preprint, arXiv:math.GT/0404311.

\bibitem{Hal}  M. Halic: {\em On the geography of symplectic 6-manifolds}, Manuscript. Math., 1999 no. 99, 371--381.

\bibitem{Ko}  M. Korkmaz: {\em Noncomplex smooth 4-manifolds with Lefschetz fibrations}, Internat. Math. Res. Not., 2001 no. 3, 115--128.

\bibitem{Li}  T.-J. Li:  {\em Smoothly embedded spheres in symplectic\/ 
$4$-manifolds}, Proc. Amer. Math. Soc., {\bf 127} (1999), 609--613.

\bibitem{LP}  M. Lopes, R. Pardini:  {\em The geography of irregular surfaces}, Math. Sci. Res. Inst. Pub. (59), Cambridge Univ. Press, Cambridge, (2012), 349--378.

\bibitem{Lu} K. M. Luttinger:  {\em Lagrangian tori in $\mathbb{R}^4$}, 
J. Differential Geom., {\bf 42} (1995), 220--228.

\bibitem{M} Y. Matsumoto: {\em Lefschetz fibrations of genus two - a topological approach}, Proceedings of the 37th Taniguchi Symposium on Topology and Teichmuller Spaces, (1996) 123--148.

\bibitem{MT} B. Moishezon, M. Teicher: {\em Simply connected algebraic surfaces of positive index}, Invent. Math, {\bf 89} (1987), 601--644.

\bibitem{MW} J. McCarthy, J. Wolfson: {\em Symplectic normal connect sum}, Topology, {\bf 33} (1994), no. 4, 729--764.

\bibitem{Nor}  M. Nori:
{\em Zariski's conjecture and related problems}, 
Ann. Sci. \'Ecole Norm. Sup. (4) {\bf 16} (1983), no. 2, 305--344.

\bibitem{Par} R. Pardini: {\em Abelian covers of algebraic varieties}, J. Reine Angew. Math. 417, 1991, 191--214.

\bibitem{JP} J. Park: {\em The geography of irreducible 4-manifolds}, Amer. Math. Soc. {\bf 126} (1998), no. 8, 2493--2503.

\bibitem{PSz}  B.~D. Park, Z. Szabo: {\em The geography problem for irreducible symplectic 4-manifolds}, Trans. Amer. Math. Soc., {\bf 352}, (2000) 3639--3650.

\bibitem{Pe} U. Persson: {\em An introduction to the geography of surfaces of general type}.  Algebraic geometry, Bowdoin, 1985 (Brunswick, Maine, 1985),  195--218, Proc. Sympos. Pure Math., 46, Part 1, Amer. Math. Soc., Providence, RI, 1987.

\bibitem{PPX} U. Persson, C. Peters, G. Xiao: {\em Geography of spin surfaces}, Topology, {\bf 35} (1996), no. 4, 845--862.

\bibitem{PS} R. Pries, K. Stevenson: {\em A survey of Galois theory of curves in characteristic $p$}, Fields Institute Communications, Volume 60, Amer. Math. Soc., (2011), 169--191.

\bibitem{RU} X. Roulleau, G. Urzua: {\em Chern slopes of simply connected complex surfaces of general type are dense in [2,3]}, Ann. Math. 

\bibitem{So} A. J. Sommese: {\em On the density of ratios of Chern numbers of algebraic surfaces}, Math. Ann., {\bf 268} (1984), no. 2, 207--221.

\bibitem{SS} S. Sakall{\i}: Ph.D. thesis.

\bibitem{Ta} C. Taubes: {\em The Seiberg-Witten invariants and symplectic forms}, Math. Res. Lett., {\bf 1} (1994), 809--822.

\bibitem{U} G. Urzua: {\em Arrangements of Curves and Algebraic Surfaces}, PhD thesis, University of Michigan, 2008. 

\bibitem{U} M. Usher: {\em Minimality and Symplectic Sums}, Internat. Math. Res. Not. Art. ID 49857, 17 pp. 

\bibitem{X} G. Xiao: {\em $\pi_1$ of elliptic and hyperelliptic surfaces}, Int. J. Math., {\bf 2} (1991), 599--615.

\end{thebibliography}
\end{document}